\documentclass[a4paper,12pt]{article}

\usepackage{amsmath}
\usepackage{amssymb}
\usepackage{slashed}
\usepackage{amsthm}

\theoremstyle{definition}
\newtheorem{definition}{Definition}[section]

\newtheorem{theorem}{Theorem}[section]
\newtheorem{corollary}{Corollary}[theorem]
\newtheorem{lemma}[theorem]{Lemma}

\newtheorem{prop}{{\bf Proposition}}[section]

\def\appendix#1{\addtocounter{section}{1}\setcounter{equation}{0}
\renewcommand{\thesection}{\Alph{section}}
\section*{Appendix \thesection\protect\indent \parbox[t]{11.15cm}{#1}}
\addcontentsline{toc}{section}{Appendix \thesection\ \ \ #1}}

\def\hn{{\hat{\nabla}}}

\counterwithout{equation}{section}

\begin{document}

\begin{center}
\vspace*{-1.0cm}
\begin{flushright}
\end{flushright}

\hfill DMUS-MP-19-10
\\

\vspace{2.0cm} {\Large \bf Eigenvalue  estimates for multi-form modified Dirac operators } \\[.2cm]

\vskip 2cm
J. Gutowski$^1$ and   G.  Papadopoulos$^2$
\\
\vskip .6cm

\begin{small}
${}^1$ \textit{Department of Mathematics, University of Surrey
\\
Guildford, Surrey, GU2 7XH, UK}\\
\texttt{j.gutowski@surrey.ac.uk}
\end{small}

\vskip5mm

\begin{small}
${}^2$ \textit{Department of Mathematics, King's College London
\\
Strand, London WC2R 2LS, UK}\\
\texttt{george.papadopoulos@kcl.ac.uk}
\end{small}

\end{center}

\vskip 2.5 cm

\begin{abstract}
\noindent

We give  estimates for the eigenvalues of  multi-form modified Dirac operators which are constructed from a standard Dirac operator with the addition of
 a Clifford algebra element associated to a multi-degree form. In particular  such estimates are presented for  modified Dirac operators with a $k$-degree form   $0\leq k\leq 4$, those modified with  multi-degree $(0,k)$-form $0\leq k\leq 3$ and the horizon Dirac operators which are modified with a multi-degree $(1,2,4)$-form. In particular, we give the necessary geometric conditions  for such operators to admit zero modes as well as those for  the zero modes to be parallel with a respect to a suitable connection.  We also  demonstrate that   manifolds which admit such parallel spinors  are associated with twisted covariant form hierarchies which generalize the conformal Killing-Yano forms.

\end{abstract}

\newpage

\renewcommand{\thefootnote}{\arabic{footnote}}

\section{Introduction}

It is a consequence of the Lichnerowicz formula and theorem that the Dirac operator on  closed\footnote{Closed manifolds are those that are compact and without boundary.} spin manifolds $M^n$ with scalar curvature $R\gneq 0$, i.e. $R\geq 0$ with $R\not=0$ somewhere, does not admit zero modes.
As a result, the Atiyah-Singer index of the Dirac operator vanishes. Moreover if $R=0$, then
all zero modes of the Dirac operator are parallel.
It is apparent from the statements above that the vanishing of the index of the Dirac operator is  a necessary condition for the existence  on $M^n$ of metrics with positive scalar curvature. In addition, Gromov and Lawson have shown that the vanishing of
the real index of the Dirac operator now acting   on the sections of a Clifford bundle is the necessary and sufficient condition
for the existence of a  metric with positive scalar curvature on simply connected closed spin  manifolds  $M^n$, $n\geq 5$, see e.g. \cite{lawson}.
Other key results that are centred around the Lichnerowicz formula and theorem are the  estimates for the eigenvalues of the Dirac operator.  These are of two kinds, one  is a lower bound on the  eigenvalues of the Dirac operator, see \cite{fried, hijazi, tfbook}, and other is an upper bound for the growth  of the eigenvalues of the Dirac operator \cite{witten, atiyah}.
All these results summarize some of the applications of spinors to geometry.

More recently a key ingredient in the proof of the  horizon conjecture \cite{iibhor} in several supergravity theories is   a generalisation
of the Lichnerowicz formula and theorem for  connections of a  spin bundle which are not induced from the Levi-Civita connection
of the underlying manifold, see e.g. \cite{iibhor, 5hor, mhor, revhor}. Such connections have the general form ${{\hat{\nabla}}}_X=\nabla_X+\Sigma_X$, where $\nabla$ is the  Levi-Civita spin connection
and $\Sigma_X$ is a Clifford algebra element constructed from the form field strengths of these theories.
Typically
the holonomy of these connections is a subgroup of a general linear group and therefore they do not a priori have special geometric properties like for example  preserving the invariant inner
product on the spin bundle. Nevertheless, some of the results of the Lichnerowicz formula and theorem can be generalized to these connections including the
relation between the zero modes of an  operator ${\cal{D}}$ and the parallel spinors of ${{\hat{\nabla}}}$, where ${\cal{D}}$ is constructed from the Dirac operator upon addition of a Clifford algebra element associated to a multi-form. So far, the existence of such connections and multi-form Dirac operators lies within the realm of supergravity theories and it is not
a priori apparent the extent to which such connections exist on every spin (or spin$_c$) Riemannian manifold.

Motivated by this, let us consider the definition; see appendix A for the notation.
\begin{definition}
Let $(M^n,g)$ be a Riemannian spin manifold.  A  multi-form modified Dirac operator, or multi-form Dirac operator for short,  ${\cal{D}}$ is an operator of the form ${\cal{D}}=\slashed{\nabla}+\slashed \omega$, where $\slashed{\nabla}$ is the Dirac operator constructed
from the spin Levi-Civita connection of $M^n$, $\omega$ is a multi-degree form on $M^n$ and $/~:~ \Lambda^*(M^n)\rightarrow {\rm Cl}(M^n)$ is the bundle
isomorphism between the bundle of forms $\Lambda^*(M^n)$  and the Clifford algebra bundle ${\rm Cl}(M^n)$ over $M^n$, respectively.
\end{definition}

The purpose of this paper is to initiate a more systematic investigation of  spin bundle connections ${{\hat{\nabla}}}$ and multi-form Dirac operators ${\cal{D}}$ on a spin Riemannian manifold for which a Lichnerowicz type of formula and theorem can be formulated,  and to explore some applications
to geometry.  One of the results is the existence of estimates for the eigenvalues of  such multi-form Dirac operators. These results also provide
 an alternative proof for the estimates for the eigenvalues of the Dirac operator given  in \cite{fried, hijazi, tfbook}.

One class of ${\cal{D}}$ operators that we shall be considering are those constructed from a single $k$-degree form $\omega$. Provided that $k\leq 3$ and under certain conditions on $\omega$, one can show
 that the norm of all eigenvalues of ${\cal{D}}$ is bounded from below. Furthermore, one can show that under certain conditions  on the geometry of $M^n$, the zero modes of ${\cal{D}}$ are parallel with respect to
\begin{eqnarray}
{{\hat{\nabla}}}_X=\nabla_X+k_1\, \slashed{\omega}\cdot \slashed{X}+k_2\, i_X\slashed{\omega}~,
\label{omegacon}
 \end{eqnarray}
 where $\cdot$ denotes the Clifford algebra multiplication, $i_X\slashed{\omega}\mathrel{\vcenter{\baselineskip0.5ex \lineskiplimit0pt
                     \hbox{\scriptsize.}\hbox{\scriptsize.}}}
                    =/(i_X\omega)$ and $k_1, k_2\in \mathbb{C}$.    The main results are described in the
 theorems and propositions \ref{th:1form}, \ref{prop1}, \ref{th:2fc}, \ref{prop2},  \ref{th:3f} and  \ref{prop3}.

The geometry of manifolds that admit ${{\hat{\nabla}}}$ parallel spinors has also been investigated. It is demonstrated that all such manifolds admit
a twisted covariant form hierarchy, for the definition see \ref{twis}.  This generalizes the notion of conformal Killing-Yano forms.
The conditions satisfied by the forms of the twisted covariant hierarchy  can be found in the corollary \ref{c:1ff} and propositions \ref{p:2fb} and \ref{p:3fb}.

Another class of ${\cal{D}}$ operators that will be investigated are those constructed from degree $(0,k)$ multi-forms.  It is shown that a fundamental formula that leads to a Lichnerowicz type of theorem and estimates for the
eigenvalues of ${\cal{D}}$ can be constructed provided that $1\leq k\leq 3$.  Our main results are described in the theorem and propositions \ref{theorem01}, \ref{prop01},
\ref{theorem02}, \ref{prop02}, \ref{03th}, \ref{prop03} and \ref{c:03f}.  The twisted covariant form hierarchy associated with ${{\hat{\nabla}}}$-parallel spinors in this case can be constructed from those associated with 0-form and $k$-form connections ${{\hat{\nabla}}}$.
Furthermore, we present the fundamental formula and explore some of its consequences for the horizon Dirac operators which are examples of  (1,2,4)-form modified Dirac operator.

This paper is organized as follows. In section two, we review the Lichnerowicz formula and theorem and present our results for 0-form modified
Dirac operators.  In sections 3, 4 and 5, the fundamental formulae are described  and their consequences are explored for
1-form, 2-form and 3-form modified Dirac operators, respectively. We also demonstrate that the construction cannot be extended to $k>3$-form Dirac operators.
In sections 6, 7 and 8, we give fundamental formulae for (0,$k$)-form Dirac operators, $k=1,2,3$, respectively. In section 9, we extend our work to the horizon Dirac operators and in section 10, we give our conclusions. In appendix A, we have collected our conventions.

\section{ 0-form modified Dirac operators}

 \subsection{Revisiting Lichnerowicz formula and theorem }

 The  Lichnerowicz formula relates the square of the Dirac operator $\slashed \nabla$ to the Laplace operator of the Levi-Civita connection and to the scalar curvature $R$ of the underlying Riemannian manifold $(M, g)$ as
 \begin{eqnarray}
 \slashed \nabla^2=\nabla^2- R/4~.
  \end{eqnarray}
  Using this after a straightforward calculation one finds the ``fundamental identity''
  \begin{eqnarray}
  \nabla^2\parallel\epsilon \parallel^2=2 \parallel\nabla\epsilon\parallel^2+2\,{\rm Re}\langle\epsilon, \slashed \nabla^2\epsilon\rangle+{1\over2} R \parallel\epsilon\parallel^2~.
  \label{diracmax}
  \end{eqnarray}
  Assuming that $M^n$ is closed,  $R\geq0$ and $\slashed\nabla\epsilon=0$, it is a consequence of the  Hopf maximum principle that the differential
  equation on $\parallel\epsilon \parallel^2$ will have no solutions.  This is a contradiction unless $\slashed \nabla$ does not have zero modes.  On the other hand
  if $\slashed \nabla\epsilon=0$ and $R=0$, the Hopf maximum principle will imply that $\nabla\epsilon=0$.  Therefore all zero modes of the Dirac operator are parallel, i.e.
  \begin{eqnarray}
\nabla_X\epsilon=0\Longleftrightarrow \slashed \nabla \epsilon=0~.
\end{eqnarray}

Assuming that $M^n$ is closed, integrating (\ref{diracmax}) on $M^n$ one derives the more familiar integral identity
\begin{eqnarray}
\int_{M^n} \parallel\slashed \nabla\epsilon\parallel^2=\int_{M^n} \parallel\nabla\epsilon\parallel^2+ {1\over4}\int_{M^n} R \parallel \epsilon\parallel^2~.
\label{intform}
\end{eqnarray}
The two consequences of the Lichnerowicz formula explained above via the use of the fundamental identity (\ref{diracmax}) can also be derived from this integral identity.
Moreover, if $\epsilon$ is an eigenspinor with eigenvalue $\lambda$, one can easily conclude that
\begin{eqnarray}
|\lambda|^2\geq{1\over4} \inf_{M^n}\, R~,
\end{eqnarray}
i.e. all eigenvalues of the Dirac operator are bounded below by the minimum of the scalar curvature on $M^n$.
These results can be summarized as follows.

\begin{theorem}
Let $(M^n, g)$ be a closed spin Riemannian manifold.  Then
\begin{enumerate}
\item If the Ricci scalar $R\gneq 0$ on $M^n$, then ${\rm Ker}\,\slashed\nabla=\{0\}$.

\item If $R=0$, then ${\rm Ker}\,\slashed\nabla={\rm Ker}\,\nabla$ and all the zero modes of $\slashed\nabla$ are parallel spinors.

\item Let $\lambda$ be any eigenvalue of $\slashed\nabla$, then $|\lambda|^2\geq{1\over4} \inf_{M^n}\, R$.
\end{enumerate}
\end{theorem}

Some applications of the above theorem have already been mentioned in the introduction.

\subsection{Fundamental  identities   for 0-form Dirac operators }

As ${\cal{D}}$ is a 0-form Dirac operator choose
 $\omega=e f$ where $f$ is a real smooth function on $M^n$ and $e\in \mathbb{C}$. So ${\cal{D}}=\slashed \nabla+ ef$.

\begin{prop}
\label{0formpropextra}
A fundamental identity for ${\cal{D}}$ is
\begin{eqnarray}
&&\nabla^2\parallel\epsilon\parallel^2=\left({1\over2} R+ c_1 f^2 \right)\parallel\epsilon\parallel^2+2 \parallel{{\hat{\nabla}}}\epsilon\parallel^2
\cr
&&
 +2 {\rm Re} \langle\epsilon, (\slashed {\nabla}-2\bar k f+\bar e f){\cal{D}}\epsilon\rangle
-2 \nabla_i{\rm Re}\langle \epsilon, e f \Gamma^i\epsilon\rangle~,
\label{0max}
\end{eqnarray}
where ${{\hat{\nabla}}}_X=\nabla_X+ k f \slashed{X}$, $c_1=-2 |k|^2 n+2{\rm Re} ((2\bar k+e) e)-4({\rm Re}\,e)^2$ and  $k\in\mathbb{C}$.
\end{prop}
\begin{proof}
  Some of the steps for the derivation of the fundamental identity are as follows. Observe that (\ref{diracmax}) can be rewritten as
\begin{eqnarray}
\nabla^2\parallel\epsilon \parallel^2&=&2 \parallel({{\hat{\nabla}}}-\Sigma)\epsilon\parallel^2+2\,{\rm Re}\langle\epsilon, \slashed \nabla({\cal{D}}-\slashed\omega)\epsilon\rangle+{1\over2} R \parallel\epsilon\parallel^2
\cr
&=&2 \parallel  {{\hat{\nabla}}}\epsilon\parallel^2-4{\rm Re}\langle  \Sigma\epsilon, \nabla\epsilon\rangle  -2 \parallel\Sigma\epsilon\parallel^2    +2\,{\rm Re}\langle\epsilon, \slashed \nabla{\cal{D}}\epsilon\rangle
\cr
&&
-2\,{\rm Re}\langle\epsilon, \slashed \nabla(\slashed\omega\epsilon)\rangle  +{1\over2} R \parallel\epsilon\parallel^2    ~.
\end{eqnarray}
So far this calculation applies to all multi-form Dirac operators.  In this particular case, one finds that
\begin{eqnarray}
{\rm Re}\langle  \Sigma\epsilon, \nabla\epsilon\rangle={\rm Re}\langle  \epsilon, \bar kf \slashed\nabla\epsilon\rangle=
{\rm Re}\langle  \epsilon, \bar kf ({\cal{D}}-e f)\epsilon\rangle~,
\end{eqnarray}
and
\begin{eqnarray}
{\rm Re}\langle\epsilon, \slashed \nabla(\slashed\omega\epsilon)\rangle&=&{\rm Re}\langle\epsilon, (e \slashed df\epsilon+ ef \slashed\nabla\epsilon)\rangle
=\nabla_i{\rm Re}\langle\epsilon, e f\Gamma^i \epsilon\rangle
\cr &&-{\rm Re}\langle\epsilon, (e+\bar e) f({\cal{D}}-ef) \epsilon\rangle
+{\rm Re}\langle\epsilon, (ef {\cal{D}}-e^2 f^2) \epsilon)\rangle~,
\end{eqnarray}
which together with $\parallel\Sigma\parallel^2=|k|^2 n \parallel\epsilon\parallel^2$ give the fundamental identity of the proposition.
\end{proof}

{\bf Remark:}  The constants $e$ and $k$ and others that will be introduced later are referred to as the parameters of the problem. Many of the statements that follow require  certain
relations between these parameters which will increasingly  become more involved. One of the tasks is to verify that there is a range in the space of parameters
for which the statements stated hold.

\begin{theorem}\label{0formth}
Let $M^n$ be a closed spin manifold, and let
$c_1$ be the constant given in Proposition ({\ref{0formpropextra}}).
Consequences  of the fundamental identity are the following.
\begin{enumerate}

\item If ${1\over2} R+ c_1 f^2\gneq 0$ on $M^n$, then ${\rm Ker}\,{\cal{D}}=\{0\}$~,

\item If ${1\over2} R+ c_1 f^2=0$,  $f\not=0$ at some point on $M^n$ and ${\rm Ker}\,{\cal{D}}\not=\{0\}$, then ${\rm Ker}\,{\cal{D}}= {\rm Ker}\,{{\hat{\nabla}}}$ and $e=nk$~.

\end{enumerate}

\end{theorem}

\begin{proof}
Indeed let $\epsilon$ be a zero mode of ${\cal{D}}$, ${\cal{D}}\epsilon=0$. After integrating the  fundamental identity on $M^n$, one finds that
\begin{eqnarray}
\int_M \left(\left({1\over2} R+ c_1f^2\right)\parallel\epsilon\parallel^2+2 \parallel{{\hat{\nabla}}}\epsilon\parallel^2\right)=0~,
\end{eqnarray}
which leads to a contradiction provided that ${1\over2} R +c_1f^2\gneq 0$.  This establishes the first part of the statement.

On the other hand if ${\cal{D}}\epsilon=0$ and ${1\over2} R +c_1f^2= 0$, then one concludes that $\parallel{{\hat{\nabla}}}\epsilon\parallel^2=0$ and so ${{\hat{\nabla}}}\epsilon=0$.  Thus
all the zero modes of ${\cal{D}}$ are parallel.  Therefore ${\rm Ker}\,{\cal{D}}\subseteq {\rm Ker}\,{{\hat{\nabla}}}$.

On the other hand ${{\hat{\nabla}}}\epsilon=0$ implies that $\slashed\hn \epsilon=0$ and so $({\cal{D}}-\slashed\hn)\epsilon=f(e-nk) \epsilon=0$.
The spinor $\epsilon$ is parallel with respect to ${{\hat{\nabla}}}$ and so it is nowhere zero on $M^n$. In addition as $f$ does not vanish everywhere on $M^n$, one
finds that $e=nk$ and so  ${\cal{D}}=\slashed\hn$. As a result ${\rm Ker}{{\hat{\nabla}}}\subseteq {\rm Ker}{\cal{D}}$.  Combining this with ${\rm Ker}{\cal{D}}\subseteq {\rm Ker}{{\hat{\nabla}}}$ above,
one concludes that ${\rm Ker}{\cal{D}}= {\rm Ker}{{\hat{\nabla}}}$ which  establishes the second part of the statement.
\end{proof}

One of the applications of the fundamental identity is the derivation of an  estimate  for the eigenvalues of the Dirac operator.
\begin{corollary}
Let $M^n$ be a closed spin manifold.
Let $\lambda$ be an  eigenvalue of the Dirac operator $\slashed\nabla$ and $n>1$, then
\begin{eqnarray}
|\lambda|^2\geq   {1\over 4 (s^2 n-2s+1)} \inf_{M^n} R~,
\label{s0bound}
\end{eqnarray}
where $s\in \mathbb{R}$. An upper lower bound for $|\lambda|^2$ is attained for $s=n^{-1}$.
\end{corollary}

\begin{proof}   To derive this from the theorem \ref{0formth} take $f=1$ and choose $e=-\lambda$ and $k=-s \lambda$.  In such a  case, one has   ${\cal{D}}=\slashed \nabla-\lambda$. As $\slashed\nabla$ is formally anti-self-adjoint, $\lambda$ is imaginary.  If $\epsilon$ is an eigen-spinor of $\slashed \nabla$ with eigen-value $\lambda$, then
${\cal{D}}\epsilon=0$.  As ${\cal{D}}$ has zero modes, it is required that
\begin{eqnarray}
{1\over2} R +c_1f^2&=&{1\over2} R +2|e|^2 (-s^2 n+2s-1)
\nonumber \\
&=&{1\over2} R +2|\lambda|^2 (-s^2 n+2s-1)\leq 0~.
\end{eqnarray}
As for $n>1$, $s^2 n-2s+1>0$,
this inequality can be re-arranged  to bound   the eigenvalues of $\slashed \nabla$ in (\ref{s0bound}).  The minimum of polynomial $s^2 n-2s+1$
is at $s=n^{-1}$ and therefore an  upper lower bound for the eigenvalues of $\slashed\nabla$ is
\begin{eqnarray}
|\lambda|^2\geq   {n\over 4( n-1)} \inf_{M^n} R~.
\end{eqnarray}
This bound has originally been established  by Friedrich    \cite{fried}.
\end{proof}

{\bf Remark:}  As the focus is on  ${\cal{D}}$, ${{\hat{\nabla}}}$ is chosen to optimise the properties of ${\cal{D}}$. In this respect
it is allowed to vary $k$ in such a way that the fundamental identity can effectively be used. We have already seen examples
of this above.  A similar approach leads to estimates for the eigenvalues  $\lambda$ of ${\cal{D}}$.  As ${\cal{D}}$ is not formally anti-self-adjoint for $e$ complex, the eigenvalues
in general are complex numbers.

\begin{prop}
\label{prop0}
Let $M^n$ be a closed spin manifold.
Suppose that $k=e$ and take $\lambda$ to be an eigenvalue of ${\cal{D}}$, then
\begin{eqnarray}
|\lambda|^2\geq \inf_{M^n} \left({1\over4} R +(1-n)  |e|^2 f^2 \right)~.
\end{eqnarray}
\end{prop}

\begin{proof} For $k=e$, the fundamental identity can be re-arranged as
\begin{eqnarray}
&&\nabla^2\parallel\epsilon\parallel^2=\left({1\over2} R +2(1-n)  |e|^2 f^2 \right)\parallel\epsilon\parallel^2+2 \parallel{{\hat{\nabla}}}\epsilon\parallel^2
\cr
&&- 2 \parallel {\cal{D}}\epsilon\parallel^2+2 \nabla_i {\rm Re} \langle\epsilon, \Gamma^i {\cal{D}} \epsilon\rangle
-2 \nabla_i{\rm Re}\langle \epsilon, e f \Gamma^i\epsilon\rangle~.
\label{0max}
\end{eqnarray}
Integrating this  over $M^n$, one finds that
\begin{eqnarray}
\int_M \parallel {\cal{D}}\epsilon\parallel^2=\int_M \left({1\over4} R +(1-n)  |e|^2 f^2 \right)\parallel\epsilon\parallel^2+\int_M \parallel{{\hat{\nabla}}}\epsilon\parallel^2~.
\label{0id}
\end{eqnarray}
Choosing $\epsilon$ to be an eigen-spinor of ${\cal{D}}$ with eigen-value $\lambda$, one establishes the proposition.
\end{proof}

{\bf Remark:}  If ${1\over4} R +(1-n)  |e|^2 f^2=0$, the identity (\ref{0id}) does not imply that the eigenspinors $\epsilon$ are parallel with respect
to some connection $\check\nabla_X={{\hat{\nabla}}}_X+p f \slashed X$ unless $f$ is constant.

\subsection{Manifolds with  ${{\hat{\nabla}}}$-parallel spinors}

The geometry of manifolds admitting ${{\hat{\nabla}}}$-parallel spinors, ${{\hat{\nabla}}}\epsilon=0$, is restricted.  Here the main focus is to investigate some aspects
of the geometry of manifolds that satisfy the conditions as stated in the second part of the theorem \ref{0formth}.

\begin{prop}
For a generic choice of metric on $M^n$, $f$ and $k$, the Lie algebra of the holonomy group  of ${{\hat{\nabla}}}$ is contained in $\mathfrak{pin}\otimes \mathbb{C}$,
$\mathfrak{Lie}\, {\rm Hol}({{\hat{\nabla}}})\subseteq \mathfrak{pin}\otimes \mathbb{C}$.
\end{prop}
\begin{proof}
As a consequence of the Ambrose-Singer theorem the Lie algebra of the the holonomy group ${\rm Hol}({{\hat{\nabla}}})$ is spanned by  $\hat R(X,Y)$ evaluated at every point of manifold $M^n$, where $\hat R$ is the curvature of
${{\hat{\nabla}}}$ and $X,Y$ are any two vector fields. A short computation using the formula $\hat R(X,Y)={{\hat{\nabla}}}_X {{\hat{\nabla}}}_Y-{{\hat{\nabla}}}_Y{{\hat{\nabla}}}_X-{{\hat{\nabla}}}_{[X,Y]}$ reveals that
\begin{eqnarray}
 \hat R(X,Y)= R(X,Y)+ k X(f) \slashed Y-k Y(f) \slashed X+  k^2 f^2 (\slashed X \slashed Y-\slashed Y \slashed X)~,
\end{eqnarray}
where $R(X,Y)$ is the curvature of the frame connection induced on the spin bundle, i.e in an orthonormal co-frame basis $\{e_i; i=1,\dots, n\}$, one has $R(e_i,e_j)={1\over4}R_{ij,kl} \Gamma^k \Gamma^l$, see also appendix A.

Viewing ${\rm Cl}_n$ as a Lie algebra with commutator constructed from the Clifford multiplication in the standard way, i.e. $[a,b]= a\cdot b-b\cdot a$,
at every point $\hat R(X,Y)$ spans ${\rm Cl}^1_n\oplus {\rm Cl}^2_n$, where the superscript labels refer to the $\mathbb{Z}$ grading of the Clifford algebras.  This is the Lie algebra $\mathfrak{pin}$ of the Pin group. As the coefficients are complex,
one concludes that $\mathfrak{Lie}\, {\rm Hol}({{\hat{\nabla}}})\subseteq \mathfrak{pin}\otimes \mathbb{C}$.
\end{proof}

\begin{prop}\label{curv0}
If $M^n$ admits a ${{\hat{\nabla}}}$-parallel spinor $\epsilon$, then $\hat R(X,Y)\epsilon=0$ and
\begin{eqnarray}
\label{inteq1}
&&\left(-{1\over2} i_Y\slashed R+ n k i_Yd f-k \slashed\partial f \cdot \slashed Y- 2  (n-1) k^2 f^2\slashed Y \right)\epsilon=0~,
\end{eqnarray}
\begin{eqnarray}
\label{inteq2}
&&\left(-{1\over2} R+2 (n-1) k \slashed\partial f-2 n (n-1) k^2 f^2\right)\epsilon=0~,
\end{eqnarray}
where again $i_Y\slashed R\mathrel{\vcenter{\baselineskip0.5ex \lineskiplimit0pt
                     \hbox{\scriptsize.}\hbox{\scriptsize.}}}
                    =/(i_Y R)$ with $R$ the Ricci tensor.
\end{prop}
\begin{proof}
Clearly $\hat R(X,Y)\epsilon=0$ is a consequence of ${{\hat{\nabla}}}_X\epsilon=0$.  The remaining two identities on the Ricci tensor  and curvature scalar follow
after  some further Clifford algebra and the Bianchi identities of Riemann curvature.
\end{proof}

\begin{corollary}Let $M^n$ be a closed  spin manifold.
Suppose that $n>1$, and ${\hat \nabla} \epsilon =0$.
\begin{enumerate}
\item   If $k$  is  imaginary, $k\not=0$, then $f$ is constant.

\item If  $k$ is real, $f$ is nowhere vanishing, and
the assumptions of the second part of the theorem \ref{0formth} hold, then $k=0$.

\end{enumerate}
\end{corollary}

\begin{proof}

For the first part take the Dirac inner product of ({\ref{inteq1}}) with $\epsilon$,
and then take the imaginary part, to obtain the condition
$(n-1) k (i_Y df) \parallel \epsilon \parallel^2=0$. As $\epsilon$ is nowhere
vanishing, it follows that $f$ is constant.

For the second part, as $e=n k$, one has that $c_1=-2n (n-1) |k|^2$.  For $k$
real, the condition ({\ref{inteq2}}) implies that $(k \slashed\partial f-2n  k^2 f^2)\epsilon=0$.  As $f$ is nowhere zero, this can be rewritten as
 $k (\slashed \partial\log |f|-2n k) \epsilon=0$.  As $\epsilon$ is nowhere vanishing and $f$ has critical points, one finds that $k=0$.
\end{proof}

Given a spinor $\epsilon$ on a manifold $M^n$, one can define a set of  forms on $M^n$  as
\begin{eqnarray}
\chi_p(X_1, X_2, \dots, X_p)= {i^{[{p\over2}]} \over p!}  \sum_{\sigma}(-1)^{|\sigma|}
\langle \epsilon, \slashed X_{\sigma(1)}\cdot \slashed X_{\sigma(2)}\dots \cdot \slashed X_{\sigma(p)}\epsilon\rangle~,
\label{pbiforms}
\end{eqnarray}
where $\sigma$ is a permutation and $|\sigma|$ is the signature of $\sigma$. Using this definition, it is straightforward to demonstrate the following.

\begin{prop}
Let ${{\hat{\nabla}}}\epsilon=0$, then
\begin{eqnarray}
&&(\nabla_Y \chi_p)(X_1, X_2, \dots, X_p)=-a_{p,p+1}  (\bar k+(-1)^p k) f\, \chi_{p+1} (Y, X_1, \dots, X_p)
\cr
&&
\qquad\qquad\qquad-a_{p,p-1}
 (\bar k+(-1)^{p-1} k) f\,(\alpha_Y\wedge \chi_{p-1})(Y, X_1, \dots, X_p)~,
\end{eqnarray}
where $\alpha_Y(X)=g(Y,X)$ and $a_{p,q}= i^{[{p\over2}]-[{q\over2}]}$.
\end{prop}

 \begin{corollary}
 The p-forms $\chi_p$ are conformal Killing-Yano forms.
 \end{corollary}

 \begin{proof}
 Indeed it is straightforward to observe that
 \begin{eqnarray}
 \nabla_Y\chi_p={1\over p+1}i_Y d\chi_p-{1\over n-p+1} \alpha_Y\wedge \delta \chi_p~,
 \end{eqnarray}
 which is the definition of a conformal Killing-Yano p-form.
 \end{proof}

 {\bf Remark:}
 Clearly if $k$ is imaginary (real), then $d\chi_p=0$ for $p$ even (odd). Similarly  if $k$ is imaginary (real), then  $\delta \chi_p=0$ for $p$ odd (even).
 Furthermore for $k$ imaginary, the 1-forms are Killing, while for $k$ real the 1-forms are closed and conformal Killing.

 {\bf Remark:} Because of the Killing property of 1-forms for $k$ imaginary, the equation ${{\hat{\nabla}}}\epsilon=0$ is called the Killing spinor equation. The geometry of manifolds
 admitting Killing spinors for $f$ a constant function have been extensively investigated in the literature, for a summary see e.g. \cite{fried}.

\section{1-form modified Dirac operators}

\subsection{Eigenvalue estimates}

Suppose that $\omega=e A$, where  $A$ is  real 1-form on $M$ and $e\in \mathbb{C}$.  Thus $
{\cal{D}} = {\slashed{\nabla}}+e {\slashed{A}}$. On even dimensional manifolds, unlike the 0-form modified Dirac operators, ${\cal{D}}: \Gamma(S^\pm)\rightarrow \Gamma(S^\mp) $, where the signs denote the chirality of the spin bundle $S=S^+\oplus S^-$.  Furthermore ${\cal{D}}$ is formally  anti-self-adjoint for $e$ imaginary.  Next set
\begin{eqnarray}
{{\hat{\nabla}}}_X = \nabla_X + k_1 {\slashed{A}} \cdot \slashed X + k_2 A(X) \, ,
\end{eqnarray}
where $k_1, k_2\in \mathbb{C}$.
The proof of the fundamental identity is similar to that for the 0-form Dirac operators, and the steps will not be repeated here.

\begin{prop}\label{fun1}
The fundamental identity of 1-form Dirac operators is
\begin{eqnarray}
&&\nabla^2\parallel \epsilon\parallel^2=2\parallel{{\hat{\nabla}}} \epsilon\parallel^2+ 2{\rm Re}\langle\epsilon, (\slashed{\nabla} +(2\bar k_1+e)\slashed{A}) {\cal{D}}\epsilon\rangle
\cr
&&+2 {\rm Re}(2 k_1+k_2+e)\delta\left( A  \parallel\epsilon\parallel^2\right)+{1\over2} R \parallel\epsilon\parallel^2-2 {\rm Re}(2 k_1+k_2) \delta A \parallel\epsilon\parallel^2
\cr
&&
+c_1 A^2 \parallel\epsilon\parallel^2+4 {\rm Im} (e+2\bar k_1+\bar k_2)\, {\rm Im}\langle\epsilon, \nabla_A\epsilon\rangle
- {\rm Re}\langle\epsilon, e \slashed dA\epsilon\rangle~,
\label{1max}
\end{eqnarray}
where $c_1=-(2n |k_1|^2 +2 |k_2|^2+ 2{\rm Re} (2\bar k_2 k_1+2\bar k_1 e+ e^2))$.
\end{prop}

One of the consequences of the fundamental identity is the following theorem.

\begin{theorem}\label{th:1form}
Let $M^n$ be a closed spin manifold.
Suppose that $A$ is closed 1-form, $dA=0$, and ${\rm Im} (e+2\bar k_1+\bar k_2)=0$, then the following hold.

\begin{enumerate}
\item If ${1\over2} R-2 {\rm Re}(2 k_1+k_2)\delta A +c_1 A^2\gneq 0$ on $M^n$, then ${\rm Ker}\, {\cal{D}}=\{0\}$.

\item If ${1\over2} R-2 {\rm Re}(2 k_1+k_2)\delta A +c_1 A^2=0$,  $A$ is non-vanishing at some point on $M^n$ and ${\rm Ker}\,{\cal{D}}\not=\{0\}$, then ${\rm Ker}\, {\cal{D}}={\rm Ker}\,{{\hat{\nabla}}}$,   $e=(2-n)k_1+k_2$ and  $c_1=-(2(n^2-n)|k_1|^2+2 |e|^2+4n {\rm Re} (\bar e k_1)+2 {\rm Re}\,e^2)$.

\end{enumerate}

\end{theorem}

\begin{proof}
Assuming that ${\cal{D}}\epsilon=0$, $dA=0$, and ${\rm Im} (e+2\bar k_1+\bar k_2)=0$, one finds that the first part of the statement follows after integrating the fundamental identity (\ref{1max}) over $M^n$.  As the only restriction on the parameters is ${\rm Im} (e+2\bar k_1+\bar k_2)=0$, the first part of the statement is valid for a 5-parameter family.

Next assuming in addition that
${1\over2} R-2 {\rm Re}(2 k_1+k_2)\delta A +c_1 A^2=0$, one   concludes   after integrating over $M^n$ that ${{\hat{\nabla}}}\epsilon=0$.  Therefore, ${\rm Ker}\, {\cal{D}}\subseteq {\rm Ker}{{\hat{\nabla}}}$.  In turn this implies that $\slashed\hn\epsilon=0$ and so ${\cal{D}}\epsilon-\slashed\hn\epsilon=(e-(2-n) k_1-k_2)\slashed A\epsilon=0$.
Therefore $(e-(2-n) k_1-k_2) A^2\epsilon=0$. As $\epsilon$ is parallel, it is no-where zero on $M^n$.  Furthermore as $A$ is non-vanishing at some point on $M^n$,
one concludes that $e=(2-n) k_1+k_2$. From this follows that ${\cal{D}}=\slashed\hn$ and ${\rm Ker}\, {\cal{D}}={\rm Ker}{{\hat{\nabla}}}$ as all ${{\hat{\nabla}}}$-parallel spinors are zero modes of
$\slashed\hn$. Solving $e=(2-n) k_1+k_2$ for $k_2$ and substituting in the expression for $c_1$ in proposition \ref{fun1}, one arrives at the expression stated in the second part of the
theorem.

Note that the second part of the statement is valid provided that the parameters are restricted as ${\rm Im} (e+2\bar k_1+\bar k_2)=0$ and $e=(2-n) k_1+k_2$.
These conditions can be re-expressed as $e=(2-n) k_1+k_2$ and ${\rm Im}\, k_1=0$.  Therefore there is a 3-parameter family of parameters that the statement is valid.
\end{proof}

{\bf Remark:}  The fundamental identity can be used in different ways.  In particular take  $A$ to be harmonic. If ${\rm Im} (e+2\bar k_1+\bar k_2)=0$ and ${1\over2} R +c_1 A^2=0$, it will be a consequence
of the Hopf maximum principle that every zero mode $\epsilon$ of ${\cal{D}}$ is parallel and  $\parallel \epsilon\parallel$ is constant. The latter condition on the length of $\epsilon$
is essential in the proof of the horizon conjecture. In particular, it is this arrangement of the fundamental identity that arises in physics-not only in this case
but also in all applications found for multi-form modified Dirac operators.

{\bf Remark: } The fundamental identity can be used to derive a bound for the eigenvalues of the Dirac operator $\slashed\nabla$ as in \cite{ hijazi}. However, this issue is left
to be examined later in the section describing  the multi-form modifications of the Dirac operator.

\begin{prop}
\label{prop1}
Let $M^n$ be a closed spin manifold.
If $dA=0$, ${\rm Im} (e+2\bar k_1+\bar k_2)=0$ and $k_1=-{\rm Re}\, e$, the eigenvalues $\lambda$ of ${\cal{D}}$ are bounded as
\begin{eqnarray}
|\lambda|^2\geq \inf_{M^n} \left({1\over4} R+{\rm Re}\,(2 e-k_2)\,\delta A+{c_1\over2} A^2\right)
\end{eqnarray}
where $c_1=-2 \left((n-2) ({\rm Re}\,e)^2+ |k_2|^2-2 {\rm Re}\,e\,{\rm Re}\,k_2+ {\rm Re}\,(e^2)\right)$.
\end{prop}
\begin{proof}
If $k_1=-{\rm Re}\, e$, $dA=0$ and ${\rm Im} (e+2\bar k_1+\bar k_2)=0$ the fundamental identity can be re-arranged as
\begin{eqnarray}
\nabla^2\parallel \epsilon\parallel^2&=& 2 \parallel{{\hat{\nabla}}}\epsilon\parallel^2-2 \parallel{\cal{D}}\epsilon\parallel^2+ 2 \nabla_i{\rm Re}\langle\epsilon, \Gamma^i{\cal{D}}\epsilon\rangle-2 {\rm Re}(e-k_2)\delta\left( A  \parallel\epsilon\parallel^2\right)
\cr
&&+\left({1\over2} R+2{\rm Re}\,(2 e-k_2)\,\delta A+c_1A^2\right)\parallel \epsilon\parallel^2 \ .
\nonumber \\
\end{eqnarray}
Assuming that $\epsilon$ is an eigen-spinor  of ${\cal{D}}$ with eigen-value $\lambda$ and integrating over $M^n$, one finds  the  bound on on $\lambda$ stated
in the proposition. The constant $c_1$ can be easily computed upon setting $k_1=-{\rm Re}\, e$ in the expression for $c_1$ in proposition \ref{fun1}.

Note that as $k_1$ is real, the range of parameters is described by the conditions   $k_1=-{\rm Re}\, e$ and ${\rm Im}\, k_2={\rm Im}\, e$.
So the statement is valid for a 3-parameter family.
\end{proof}

\subsection{Manifolds with ${{\hat{\nabla}}}$-parallel spinors}

\subsubsection{Holonomy of ${{\hat{\nabla}}}$ and Curvature identities}

The Lie algebra of the holonomy group of ${{\hat{\nabla}}}$ can be identified using the Ambrose-Singer theorem. In particular one has the following.

\begin{prop}
Suppose that the metric on $M^n$, $A$, $k_1$ and $k_2$ are generic,  then $\mathfrak{Lie}\,{\rm Hol}({{\hat{\nabla}}})\subseteq \mathfrak{pin}\otimes \mathbb{C}\oplus \mathbb{C}$.
Furthermore if $dA=0$, then $\mathfrak{Lie}\,{\rm Hol}({{\hat{\nabla}}})\subseteq \mathfrak{pin}\otimes \mathbb{C}$
\end{prop}
\begin{proof}
It suffices to find the curvature of ${{\hat{\nabla}}}$.  Indeed
\begin{eqnarray}
\hat R(X,Y)&=&R(X,Y)+k_1(\nabla_X\slashed A\cdot \slashed Y-\nabla_Y\slashed A\cdot \slashed X)+k_2 dA(X,Y)
\cr
&+& k_1^2 \slashed A \cdot \left( \slashed X \cdot \slashed A \cdot \slashed Y
- \slashed Y \cdot \slashed A \cdot \slashed X \right)~.
\end{eqnarray}
Observe that at every point on $M^n$, the curvature spans the subspace ${\rm Cl}_n^0\oplus{\rm Cl}_n^1\oplus {\rm Cl}_n^2$ of the Clifford algebra. As ${\rm Cl}_n^0$ commutes with ${\rm Cl}_n^1\oplus {\rm Cl}_n^2$ and  ${\rm Cl}_n^1\oplus {\rm Cl}_n^2$ spans $\mathfrak{pin}$,
$\mathfrak{Lie}\,{\rm Hol}({{\hat{\nabla}}})\subseteq \mathfrak{pin}\otimes \mathbb{C}\oplus \mathbb{C}$, where the compexification arises because the coefficients are complex.

Next observe that if $dA=0$, then at every point on $M^n$ the curvature spans the subspace ${\rm Cl}_n^1\oplus {\rm Cl}_n^2$ of the Clifford algebra.
This proves the second part of the proposition.
\end{proof}

The existence of ${{\hat{\nabla}}}$-parallel spinors implies conditions on the curvature of $M^n$.  In particular after a short computation, one finds the following.

\begin{prop}
Let ${{\hat{\nabla}}}\epsilon=0$, then $\hat R(X,Y)\epsilon=0$.
Moreover consequences of  $\hat R(X,Y)\epsilon=0$ are
\begin{eqnarray}
&&\bigg(-{1\over2} i_Y\slashed R +(2-n) k_1 \nabla_Y\slashed A-{1\over2} k_1 \slashed {dA}\cdot\slashed Y+k_1 \delta A \slashed Y +k_2 i_Y\slashed{ dA}
\cr &&
~~~~+2(n-2)k_1^2 (A^2 \slashed Y - \slashed A A(Y)) \bigg)\epsilon=0~,
\cr &&
\bigg(-{1\over2} R+ ((3-n)k_1+k_2) \slashed{dA}-2(1-n) k_1 \delta A
\cr
&&~~~~+ 2(1-n)(2-n)k_1^2 A^2\bigg)\epsilon=0~.
\end{eqnarray}
\end{prop}

\begin{corollary}
Let us make the same assumptions as those stated in the second part of theorem \ref{th:1form} and $n\not=1$. Then $A$ is required to satisfy
\begin{eqnarray}
\big({\rm Re} (2k_1+k_2)+ (1-n) k_1\big)\delta A-{1\over2} (c_1+2(1-n)(2-n)k_1^2) A^2=0 \ .
\end{eqnarray}
\end{corollary}
\begin{proof}
From theorem \ref{th:1form} the scalar curvature is restricted as ${1\over2} R-2 {\rm Re}(2 k_1+k_2)\delta A +c_1 A^2=0$. On the other hand
from the proposition above $-{1\over2} R-2(1-n) k_1 \delta A + 2(1-n)(2-n)k_1^2 A^2=0$ as $\epsilon$ is no-where vanishing.
Adding the two expressions together yields the result. The equation derived can be thought of as the field equations for $A$.
\end{proof}

\subsubsection{A twisted covariant form hierarchy }

Manifolds that admit ${{\hat{\nabla}}}$-parallel spinors are associated with twisted covariant form hierarchies which generalize the notion of
 conformal Killing-Yano forms. To describe the structure consider the definition.
\begin{definition}\label{twis}
Let $(M^n,g)$ be a Riemannian manifold equipped with some multi-form $F$. A  covariant form hierarchy   on $(M^n,g)$ twisted with $F$  is
a collection of forms $\chi_p$ which satisfy
\begin{eqnarray}
\nabla^F_Y (\{\chi_p\})= i_Y{\cal P}(F,\{\chi_p\})+ \alpha_Y\wedge {\cal Q}(F, \{\chi_p\})
\end{eqnarray}
where ${\cal P}, {\cal Q}: \Gamma(\Lambda^*(M))\rightarrow \Gamma(\Lambda^*(M))$ that depend on $F$ and $\chi_p$.    Furthermore  $\nabla^F$ is a connection in $\oplus^m\Lambda^*(M)$ constructed from the Levi-Civita
connection and $F$, which is not necessarily degree preserving, and $\alpha_Y(X)=g(Y,X)$.
\end{definition}
The definition above has been stated for real forms but it can be easily generalized in the complex case. In many examples ${\cal P}$ and $ {\cal Q}$ are constructed using algebraic operations between $F$ and $\{\chi_p\}$ like the wedge product and inner derivation.

{\bf Remark:} The main property of the covariant form hierarchy  on $(M^n,g)$ is that the skew-symmetric and trace representations of the covariant derivative of the forms $\chi_p$ with respect to $\nabla^F$
are expressed in terms of operations in the space of forms on $M^n$ such as the exterior derivative, the wedge and inner products with a reference
 multi-form $F$ and their adjoints.

{\bf Remark:}  Examples of covariant form hierarchies  are the Killing-Yano and conformal Killing-Yano forms.
Both are twisted with the form zero, $F=0$, and so $\nabla^F=\nabla$.  Furthermore in the former case ${\cal P}=(p+1)^{-1} d$ and ${\cal Q}=0$ while in the
latter ${\cal P}=(p+1)^{-1} d$ and ${\cal Q}=-(n-p+1)^{-1}\delta$.

\begin{lemma}\label{le:1f}
Consider the forms $\chi_p$ as in (\ref{pbiforms}) and take ${{\hat{\nabla}}}\epsilon=0$, then
\begin{eqnarray}
&&\nabla_Y \chi_p+ 2{\rm Re}\, (2k_1+k_2) A(Y) \chi_p=-2{\rm Im} k_1\,i_Y i_A \chi_{p+2}+2{\rm Re}\, k_1\, i_Y(A\wedge \chi_p)
 \cr && \qquad\qquad+2 {\rm Re}\, k_1\,\alpha_Y\wedge i_A \chi_p+2{\rm Im}\,k_1\,\alpha_Y\wedge A\wedge \chi_{p-2}~.
 \label{le1f}
\end{eqnarray}
\end{lemma}
\begin{proof}
Writing ${{\hat{\nabla}}}_Y=\nabla_Y+\Sigma_Y$ and expressing
the forms $\chi_p$ formally as $\chi_p=i^{[{p\over2}]}\langle \epsilon, \wedge^p\Gamma \epsilon\rangle$, one has that
\begin{eqnarray}
\nabla_Y\chi_p&=&i^{[{p\over2}]} \langle \nabla_Y \epsilon, \wedge^p\Gamma \epsilon \rangle +i^{[{p\over2}]}\langle  \epsilon, \wedge^p\Gamma \nabla_Y\epsilon\rangle
\cr
&=&-i^{[{p\over2}]}
\langle\epsilon, (\Sigma_Y^\dagger\cdot \wedge^p\Gamma+\wedge^p\Gamma\cdot \Sigma_Y)\epsilon\rangle~.
\end{eqnarray}
Substituting for $\Sigma_Y$ and after some Clifford algebra computation one finds that the right-hand-side of the above expression can be written in terms of the forms of the hierarchy. After some further re-arrangement one arrives at the expression stated in the lemma.
\end{proof}

\begin{corollary}\label{c:1ff}
The forms (\ref{pbiforms}) constructed from a ${{\hat{\nabla}}}$-parallel spinor $\epsilon$ define a twisted with respect to the form $A$ covariant form hierarchy  on $M^n$  as
\begin{eqnarray}
&&\nabla_Y \chi_p+ 2 {\rm Re}\, (2k_1+k_2)\, A(Y) \chi_p={1\over p+1}i_Y \bigg(d\chi_p+2{\rm Re}\,(2k_1+k_2)\, A\wedge \chi_p\bigg)
\cr
&&\qquad\qquad
-{1\over n-p+1}\alpha_Y\wedge \bigg(\delta\chi_p-2{\rm Re}\, (2k_1+k_2)\, i_A\chi_p\bigg)~.
\label{hi1f}
\end{eqnarray}

\end{corollary}
\begin{proof}
To prove this notice that the right-hand-side of (\ref{le1f}) is determined by the skew and trace representations of the left-hand-side of the same equation.
After replacing the right-hand-side of (\ref{le1f}) with those, one arrives at equation (\ref{hi1f}) which proves the corollary.
\end{proof}

{\bf Remark:}  Notice that if ${\rm Re}\, (2k_1+k_2)=0$, the covariant form hierarchy becomes untwisted and  $\chi_p$ are standard Killing-Yano forms.
In this region of parameter space both ${{\hat{\nabla}}}$ and $\slashed \hn$ retain a non-trivial dependence of $A$ even though (\ref{hi1f}) becomes independent of $A$.
Also observe that there is a common range of parameters that both the theorem \ref{th:1form} is valid and $\chi_p$ are Killing-Yano forms.

{\bf Remark:} Although (\ref{le1f}) implies (\ref{hi1f}), the converse is not always true. To see this consider the range ${\rm Re}\, (2k_1+k_2)=0$ of parameters and observe that (\ref{le1f}) retains its dependence on $A$ while (\ref{hi1f}) does not. So (\ref{le1f}) cannot be recovered from (\ref{hi1f}).

\section{2-form modified Dirac operators }

\subsection{Eingenvalue estimates}

Let $F$ be a real 2-form on $M$. The 2-form Dirac operator is chosen as ${\cal{D}}={\slashed{\nabla}}+e {\slashed{F}}$, where
$e\in \mathbb{C}$.  Furthermore consider the connection on the spin bundle
\begin{eqnarray}
\label{hatcon2}
{{\hat{\nabla}}}_X = \nabla_X + k_1 {\slashed{F}}\cdot  \slashed{X} + k_2 \slashed{i_XF}~,
\end{eqnarray}
where  $k_1, k_2\in \mathbb{C}$. If $e$ is imaginary, then ${\cal{D}}$ is formally anti-self-adjoint.

\begin{prop}
The fundamental identity of 2-form Dirac operators is
\begin{eqnarray}
\nabla^2 \parallel \epsilon \parallel^2 &=& 2 \parallel {{\hat{\nabla}}}\epsilon \parallel^2
+\left({1 \over 2}R  +c_1 F^2\right)\parallel \epsilon \parallel^2+ c_2 \parallel\slashed F\parallel^2
\cr
&+& 2 {\rm Re} \langle \epsilon, ({\slashed{\nabla}}+(2{\bar{k}}_1- e) {\slashed{F}}) {\cal{D}} \epsilon \rangle
\cr
&+& {\rm Re} \langle \epsilon,
(-4 {\bar{k}}_2 +16 {\bar{k}}_1-8 e) F^i{}_j \Gamma^j \nabla_i \epsilon \rangle
\cr
&-&{2\over3} {\rm Re} \langle \epsilon, e \slashed{dF} \epsilon\rangle+4 {\rm Re} \langle \epsilon, e \delta\slashed F  \epsilon \rangle~,
\label{max2}
\end{eqnarray}
where
\begin{eqnarray}
&&c_1=-32 |k_1|^2-2|k_2|^2+16 {\rm Re} (k_1\bar k_2)~,
\cr
&&c_2=-2 (n-8) |k_1|^2-4{\rm Re} (k_1\bar k_2)+{\rm Re}(e (4\bar k_1-2e))~,
\label{cc2y}
\end{eqnarray}
and for any form $\omega$, $\delta\slashed \omega=/(\delta \omega)$.
Furthermore if $k_2=4 k_1 -2\bar e$, then
\begin{eqnarray}
c_1=-8|e|^2~,~~
c_2=- \left(2n |k_1|^2-8 {\rm Re}( e k_1)-4 {\rm Re}( e \bar k_1) +2{\rm Re}( e^2)\right)~.
\label{cc2x}
\end{eqnarray}
\end{prop}

\begin{proof}
The identity (\ref{max2}) can be established after some computation similar to that performed for the 0-form Dirac operators.
Also the expressions for the constants $c_1$ and $c_2$ in (\ref{cc2x}) follow after a direct substitution of $k_2=4 k_1 -2\bar e$ in (\ref{cc2y}).
\end{proof}

There are different ways to present the consequences of  the fundamental identity depending on the condition satisfied by $F$.  Let us begin with
the case that $F$ is harmonic.

\begin{theorem}\label{th:2hform}
Let $M^n$ be a closed spin manifold.
Suppose that $F$ is harmonic, $F\not=0$, and  $k_2=4 k_1 -2\bar e$.
\begin{enumerate}

\item In addition if  ${1 \over 2}R  +c_1 F^2\gneq 0$ and $c_2\geq 0$,  then ${\rm Ker}\,{\cal{D}}=\{0\}$~.

\item Moreover if ${1 \over 2}R  +c_1 F^2= 0$, $c_2=0$, and ${\rm Ker}\,{\cal{D}}\not=\{0\}$, then ${\rm Ker}\,{\cal{D}}={\rm Ker}\,{{\hat{\nabla}}}$,
$e={2\over3}n\bar k_1-{1\over3} n k_1$ and $n\geq 9$.

\end{enumerate}

\end{theorem}

\begin{proof}
To prove the first part of the theorem,  one finds that the last three terms of the fundamental identity  (\ref{max2}) vanish from the assumptions made.
Furthermore as  $\epsilon\in {\rm Ker}\,{\cal{D}}$, the terms containing ${\cal{D}}\epsilon=0$ vanishes as well. Integrating the remaining formula over $M^n$ leads
to a contradiction due to the remaining conditions on the scalar curvature $R$ and  $c_2\geq 0$.  To see that there is a range of parameters such that $c_2\geq 0$, one uses
(\ref{cc2x}) to write
\begin{eqnarray}
c_2&=&-2n \bigg[
\big({\rm Re} \ (k_1) -{3 |e| \over n} \cos \psi \big)^2+ \big({\rm Im} \ (k_1) +{|e| \over n}\sin \psi \big)^2
\cr
&&+{|e|^2 \over n^2} \big(n(1-2 \sin^2 \psi)+8 \sin^2 \psi -9 \big)\bigg]~,
\end{eqnarray}
where $e=|e| \exp(i\psi)$.
In order for $c_2\geq 0$, one must have that
\begin{eqnarray}
\label{twob1}
n(1-2 \sin^2 \psi)+8 \sin^2 \psi -9 \leq 0~.
\end{eqnarray}
This inequality always holds if ${1 \over 2} \leq \sin^2 \psi \leq 1$,
with no condition on the value of $n$. However, if $\sin^2 \psi <{1 \over 2}$, then we find
\begin{eqnarray}
n \leq {9-8 \sin^2 \psi \over 1-2 \sin^2 \psi}~,
\end{eqnarray}
which gives an upper bound on $n$; though this upper bound tends to infinity
as $\sin^2 \psi \rightarrow {1 \over 2}$.

Next let us turn to the second part of the statement. As in the first part, the last three terms of the fundamental identity (\ref{max2}) vanish. Again as
${\cal{D}}\epsilon=0$, integrating the remaining identity over $M^n$ and using the remaining assumptions, one finds that ${{\hat{\nabla}}}\epsilon=0$.  Therefore
${\rm Ker\, {\cal{D}}}\subseteq {\rm Ker\, {{\hat{\nabla}}}}$.

To prove that ${\rm Ker\, {\cal{D}}}= {\rm Ker\, {{\hat{\nabla}}}}$, take $\epsilon\in {\rm Ker\, {\cal{D}}}$ and notice that as ${{\hat{\nabla}}}\epsilon=0$, one has that $\slashed\hn\epsilon=0$.  So $({\cal{D}}-\slashed\hn)\epsilon=(e-(n-4)k_1-k_2)\slashed F\epsilon=0$. The last condition requires  that $e=(n-4)k_1+k_2$. Indeed suppose that  $e\not=(n-4)k_1+k_2$.  This then gives that $\slashed F\epsilon=0$.   To proceed, consider the identity
\begin{eqnarray}
\slashed\nabla (\slashed F \epsilon)&=&\big({1\over3}\slashed{dF}-2\delta \slashed F+\Gamma^i\slashed F {{\hat{\nabla}}}_i-(k_2+(n-8)k_1) \slashed F^2
-4(k_2-4k_1) F^2\big)\epsilon~.
\nonumber \\
\label{hamham}
\end{eqnarray}
Upon imposing $\slashed F\epsilon=0$ and ${{\hat{\nabla}}}\epsilon=0$, on finds that $\bar e
F^2 \epsilon=0$.  As $\epsilon$ is no-where vanishing on $M^n$ because it is parallel and $e\not=0$, one concludes that $F^2=0$ and so $F=0$.
This is a contradiction as $F\not=0$. So one has that $e=(n-4)k_1+k_2$. Thus ${\cal{D}}=\slashed \hn$ and as ${\rm Ker}\, {{\hat{\nabla}}}\subseteq {\rm Ker}\, \slashed\hn$,
one establishes that ${\rm Ker}\, {\cal{D}}={\rm Ker}\, {{\hat{\nabla}}}$.

Next substituting for $k_2$ using $k_2=4 k_1 -2\bar e$, one finds the expression of $e$ in terms of $k_1$.
It remains to verify that there is a range in the parameter space such that $c_2=0$.  Indeed for $e={2\over3}n\bar k_1-{1\over3} n k_1$, one has that
\begin{eqnarray}
c_2=-{2|e|^2 \over n} \big(n(1-2 \sin^2 \psi)+8 \sin^2 \psi -9 \big)~.
\end{eqnarray}
Setting $c_2=0$, one finds that
\begin{eqnarray}
\sin^2\psi={9-n\over 8-2n}~.
\end{eqnarray}
Imposing $0\leq \sin^2\psi\leq 1$, the dimension of the manifold is restricted as $n\geq 9$.
\end{proof}

\begin{prop}
Let $M^n$ be a closed spin manifold.
Suppose that $F$ is harmonic, $k_2=4 k_1 -2\bar e$ and $k_1={\rm Re}\,e$, and $c_2\geq0$, then the eigenvalues $\lambda$ of ${\cal{D}}$ are bounded from below as
\begin{eqnarray}
|\lambda|^2\geq \inf_{M^n} \left({1 \over 4}R -4|e|^2 F^2\right)~.
\end{eqnarray}
\end{prop}
\begin{proof}
Under the assumptions above the three last terms of the fundamental identity (\ref{max2}) vanish.  Furthermore if $k_1={\rm Re}\,e$, the fundamental identity
can be re-arranged as
\begin{eqnarray}
\nabla^2 \parallel \epsilon \parallel^2 &=& 2 \parallel {{\hat{\nabla}}} \epsilon\parallel^2
+ \big({1 \over 2}R -8|e|^2 F^2 \big) \parallel \epsilon \parallel^2-2\parallel{\cal{D}}\epsilon\parallel^2
\cr
&+& 2 \nabla_i {\rm Re} \langle \epsilon, \Gamma^i {\cal{D}} \epsilon \rangle+c_2 \parallel {\slashed{F}} \epsilon \parallel^2~.
\label{3max2}
\end{eqnarray}
Assuming that $\epsilon$ is an eigenspinor of ${\cal{D}}$ with eigenvalue $\lambda$, one derives the bound on $\lambda$ after integrating (\ref{3max2}) over $M^n$
provided that $c_2\geq 0$.

It remains to verify that there is a range of parameters such that $c_2\geq 0$.  Indeed a direct substitution for $k_1$ in $c_2$ reveals that
\begin{eqnarray}
c_2=2|e|^2 +(8-2n)|e|^2 \cos^2\psi~.
\end{eqnarray}
It is straightforward to observe that it is always possible to choose  $e$ such that $c_2\geq 0$ for any $n$.
\end{proof}

There are other ways to re-arrange the fundamental identity (\ref{max2}) such that one can derive bounds for the eigenvalues of ${\cal{D}}$ by putting weaker restrictions
on $F$.  In particular, all the main statements that are valid for $F$ harmonic can be adapted to hold for $F$ closed.

\begin{theorem}\label{th:2fc}Let $M^n$ be a closed spin manifold.
Suppose   that $F$ is closed 2-form $dF=0$, $F\not=0$,  and $k_2=4k_1+2i {\rm Im}\,e$.
\begin{enumerate}
\item In addition if ${1 \over 2}R -8({\rm Im}\,e)^2 F^2\gneq 0$ and $c_2\geq 0$, then ${\rm Ker\, {\cal{D}}}=\{0\}$.
\item Moreover if ${1 \over 2}R -8({\rm Im}\,e)^2 F^2=0 $, ${\rm Im}\,e\not=0$,   $c_2=0 $ and ${\rm Ker\, {\cal{D}}}\not=\{0\}$, then ${\rm Ker\, {\cal{D}}}={\rm Ker\, {{\hat{\nabla}}}}$ and $\bar e=nk_1$.
\end{enumerate}
\end{theorem}
\begin{proof}
Using that $k_2=4k_1+2i {\rm Im}\,e$, one finds that
the fundamental identity (\ref{max2}) can  be re-arranged as
\begin{eqnarray}
\nabla^2 \parallel \epsilon \parallel^2 &=& 2\parallel {{\hat{\nabla}}} \epsilon\parallel^2
+ \big({1 \over 2}R -8({\rm Im}\,e)^2 F^2 \big) \parallel \epsilon \parallel^2-4\nabla^i{\rm Re}\langle\epsilon, e F_{ij} \Gamma^j \epsilon\rangle
\nonumber \\
&+& 2 {\rm Re} \langle \epsilon, \big({\slashed{\nabla}}+(2{\bar{k}}_1-e){\slashed{F}} \big) {\cal{D}} \epsilon \rangle
+c_2 \parallel {\slashed{F}} \epsilon \parallel^2
\nonumber \\
&-&{2\over3} {\rm Re} \langle \epsilon, e\, \slashed{dF} \epsilon\rangle
\label{4max2}
\end{eqnarray}
where
\begin{eqnarray}
c_2=- \left(2n |k_1|^2+8 {\rm Re}\, e\, {\rm Re}\,k_1-4 {\rm Re}( e \bar k_1) +2{\rm Re}( e^2)\right)~.
\end{eqnarray}

Assuming that $\epsilon\in {\rm Ker }\,{\cal{D}}$ and $dF=0$, the first part of the theorem follows after integrating (\ref{4max2}) over $M^n$. It remains to
show that there is a range of parameters such that $c_2\geq 0$.  Indeed $c_2$ can be written as
\begin{eqnarray}
c_2&=&-2n \left(({\rm Re}\, k_1-{1\over n} {\rm Re}\, e)^2+ ({\rm Im}\, k_1+{1\over n} {\rm Im}\, e)^2\right)
\nonumber \\
&+& 2{n+1\over n} |e|^2-4 |e|^2\cos^2\psi
\label{2c2}
\end{eqnarray}
For $c_2\geq 0$, it is required that
\begin{eqnarray}
\cos^2\psi\leq {n+1\over 2n}~.
\end{eqnarray}
This can always be arranged for a choice of $e$.

To prove the second part of the theorem, take $\epsilon\in {\rm Ker }\,{\cal{D}}$.  Using the assumptions stated and after integrating (\ref{4max2}) over $M^n$, one finds that $\epsilon\in {\rm Ker }\,{{\hat{\nabla}}}$.  Thus  ${\rm Ker }\,{\cal{D}}\subseteq {\rm Ker }\,{{\hat{\nabla}}}$.
Next notice that it is required that  $\bar e= nk_1$ as otherwise $F=0$.  Indeed if  $\bar e\not = nk_1$, one has that $\slashed F\epsilon=0$ as
$({\cal{D}}-\slashed\hn)\epsilon= (e-n \bar k_1) \slashed F\epsilon=0$.  An argument based on equation (\ref{hamham})  leads to the condition
\begin{eqnarray}
\delta \slashed F\epsilon+2(k_2-4 k_1) F^2\epsilon=0~.
\end{eqnarray}
The integrability condition is $(\delta F)^2= 4(k_2-4 k_1)^2 F^4$ because $\epsilon$ is nowhere vanishing. This together with $k_2=4k_1+2i {\rm Im}\,e$ gives $(\delta F)^2+16  ({\rm Im}\,e)^2 F^4=0$ which in turn gives
  $F=0$ as ${\rm Im}\,e\not=0$.

Thus $\bar e=n k_1$ and so ${\cal{D}}=\slashed\hn$. But ${\rm Ker}\, {{\hat{\nabla}}}\subseteq {\rm Ker}\, \slashed\hn$ and therefore ${\rm Ker}\, {{\hat{\nabla}}}\subseteq {\rm Ker}\,{\cal{D}}$. This together with  ${\rm Ker }\,{\cal{D}}\subseteq {\rm Ker }\,{{\hat{\nabla}}}$ derived from the fundamental identity above lead to ${\rm Ker }\,{\cal{D}}= {\rm Ker }\,{{\hat{\nabla}}}$.

It remains to demonstrate that there is a range of parameters such that $c_2=0$.  Indeed for $\bar e=n k_1$, one finds that
\begin{eqnarray}
c_2=2{n+1\over n} |e|^2-4 |e|^2\cos^2\psi~.
\end{eqnarray}
This has always a solution for any $n$ for an appropriate choice of $e$.
\end{proof}

\begin{prop}
\label{prop2}
Let $M^n$ be a closed spin manifold.
If $dF=0$, $k_1={\rm Re}\, e$ and $c_2\geq 0$, then
\begin{eqnarray}
|\lambda|^2\geq \inf_{M^n} \left({1 \over 4}R -4({\rm Im}\,e)^2 F^2\right)~,
\end{eqnarray}
where $\lambda$ is an eigenvalue of ${\cal{D}}$.
\end{prop}

\begin{proof}
For  $k_1={\rm Re}\, e$ and $dF=0$, the fundamental identity can be arranged as
\begin{eqnarray}
\nabla^2 \parallel \epsilon \parallel^2 &=& 2 \parallel {{\hat{\nabla}}} \epsilon\parallel^2
+ \big({1 \over 2}R -8({\rm Im}\,e)^2 F^2 \big) \parallel \epsilon \parallel^2-4\nabla^i{\rm Re}\langle\epsilon, e F_{ij} \Gamma^j \epsilon\rangle
\nonumber \\
&-&2\parallel{\cal{D}}\epsilon\parallel^2
+2 \nabla_i{\rm Re} \langle \epsilon,\Gamma^i {\cal{D}} \epsilon \rangle
+ c_2 \parallel {\slashed{F}} \epsilon \parallel^2~.
\label{funmax2}
\end{eqnarray}
Taking $\epsilon$ to be an eigenspinor with eigenvalue $\lambda$, assuming that $c_2\geq 0$  and integrating (\ref{funmax2}) over $M^n$, one finds a bound described in the proposition.

It remains to demonstrate that there is a range of parameters such that $c_2\geq 0$.  Indeed using the relations between the parameters, one finds that
\begin{eqnarray}
c_2=2|e|^2-2n|e|^2 \cos^2\psi~,
\end{eqnarray}
and so it is always possible to choose $e$ such that $c_2\geq 0$ for any $n$.
\end{proof}

\subsection{Manifolds with ${{\hat{\nabla}}}$-parallel spinors}

\subsubsection{Holonomy of ${{\hat{\nabla}}}$ and curvature identities}

Let us begin with identifying the Lie algebra of the holonomy group of ${{\hat{\nabla}}}$.
\begin{prop}
Let $M^n$ be a spin manifold, and let ${\hat{\nabla}}$
be given by ({\ref{hatcon2}}) where $F$ is a generic 2-form,
and $k_1, k_2$ are generic constants.
Then  $\mathfrak{Lie}\,{\rm Hol}({{\hat{\nabla}}})\subseteq ({\rm Cl}_n-{\rm Cl}_n^0) \otimes \mathbb{C}$, where
${\rm Cl}_n$ is considered as a Spin(n) module.
\end{prop}
\begin{proof}
To use the Ambrose-Singer theorem to find $\mathfrak{Lie}\,{\rm Hol}({{\hat{\nabla}}})$, the curvature of ${{\hat{\nabla}}}$ is
\begin{eqnarray}
\hat R(X,Y)&=&R(X,Y)+k_1(\nabla_X \slashed F \cdot \slashed Y-\nabla_Y \slashed F \cdot \slashed X)+ k_2 \left(\nabla_X i_Y\slashed F-\nabla_Y i_X\slashed F\right)
\nonumber \\
&+&k_2^2 \big(i_X \slashed F\cdot  i_Y\slashed F
-i_Y \slashed F\cdot  i_X\slashed F\big)+k_1^2 (\slashed F)^2 \cdot (\slashed X \slashed Y-\slashed Y \slashed X)
\nonumber \\
&+& (16 k_1^2 +4 k_1 k_2)\slashed F F(X,Y) -4 k_1 k_2 \slashed Y\cdot i_{i_XF}\slashed F
\nonumber \\
&+& (2k_1 k_2 + 4 k_1^2) \slashed F\cdot (\slashed X\cdot i_Y\slashed F-\slashed Y\cdot i_X\slashed F)
+4 k_1 k_2 \slashed X\cdot i_{i_YF}\slashed F~.
\label{curv2f}
\end{eqnarray}
Therefore at every point in $M^n$, $\hat R(X,Y)$ takes values in $({\rm Cl}^1_n\oplus {\rm Cl}^2_n\oplus {\rm Cl}^3_n\oplus {\rm Cl}^4_n)\otimes \mathbb{C}$. Demanding algebraic closure with respect to the Lie brackets induced on ${\rm Cl}_n$ from the Clifford multiplication, one derives the statement of the proposition.
\end{proof}

\begin{prop} If ${{\hat{\nabla}}}\epsilon=0$, then $\hat R(X,Y)\epsilon=0$, where $\hat R$ is given in (\ref{curv2f}) and
\begin{eqnarray}
&&\bigg(-{1\over2} R+ {2\over3} ((n-5) k_1+k_2) \slashed{dF}-(4(n-3) k_1+ 2k_2) \delta\slashed F
\cr
&&+ \big(-2 k_2^2+ 2 k_1^2 (-n^2+13 n-36)
-4 k_1 k_2 (n-6)\big)(\slashed F)^2
\cr
&&+ \big(-6 k_2^2 +(56-8n)k_1k_2 +32(n-4)k_1^2\big) F^2\bigg)\epsilon=0~.
\label{r2f}
\end{eqnarray}
\end{prop}
\begin{proof}
The condition $\hat R(X,Y)\epsilon=0$ follows as an integrability condition of ${{\hat{\nabla}}}\epsilon=0$. The second condition involving the scalar curvature $R$ follows
from $\hat R(X,Y)\epsilon=0$ after acting with further Clifford algebra operations.
\end{proof}

{\bf Remark:}  Unlike the previous two cases of 0- and 1-form modified Dirac operators, the condition (\ref{r2f}) is a Clifford algebra condition on $\epsilon$ and it cannot be
compared in a straightforward way with the condition ${1 \over 2}R -8({\rm Im}\,e)^2 F^2\gneq 0$ that appears in theorem \ref{th:2fc}.  Nevertheless, it is clear that it imposes
conditions on $F$ required for the existence of a parallel spinor $\epsilon$.  There is also a similar condition involving the Ricci tensor which arises from $\hat R(X,Y)\epsilon=0$. Although this has been
computed,  it has not been included in the proof.

\subsubsection{A covariant twisted form  hierarchy}

The geometry of manifolds with ${{\hat{\nabla}}}$-parallel spinors is characterized by the existence of a twisted covariant form hierarchy. The twisting form is $F$ and
the forms $\chi_p$ of the hierarchy are given in (\ref{pbiforms}).

\begin{prop}\label{p:2fb}
The twisted covariant form hierarchy of a manifold admitting a ${{\hat{\nabla}}}$-parallel spinor $\epsilon$, ${{\hat{\nabla}}}\epsilon=0$ satisfies the formula
\begin{eqnarray}
&&\nabla_Y \chi_p+ a_{p,p+1} \big(\bar k_2-4\bar k_1+(-1)^p (k_2-4 k_1)\big) i_{i_Y F} \chi_{p+1}
\cr &&~~~~
+a_{p,p-1}\big(\bar k_2-4\bar k_1+(-1)^{p-1}(k_2-4k_1)\big) i_Y F\wedge \chi_{p-1}
\cr
&&
=  {1\over p+1} i_Y\bigg(d\chi_p-a_{p,p+1} \big(\bar k_2-4\bar k_1+(-1)^p (k_2-4k_1)\big) i_F \chi_{p+1}
\cr
&&+ 2 a_{p, p-1}
\big(\bar k_2-4\bar k_1+(-1)^{p-1} (k_2-4k_1)\big) F\wedge \chi_{p-1}\bigg)
\cr
&&
-{1\over n-p+1} \alpha_Y\wedge \bigg(\delta \chi_p+ 2 a_{p,p+1}\big(\bar k_2-4\bar k_1+(-1)^p (k_2-4k_1)\big) F\vee \chi_{p+1}
\cr
&&+ a_{p, p-1}(\bar k_2-4\bar k_1+(-1)^{p-1} (k_2-4k_1))
i_F \chi_{p-1}\bigg)~,
\label{fh2f}
\end{eqnarray}
where $\vee$ is the adjoint of $\wedge$ with respect to the (weighted) inner product in the space of forms, see Appendix A.
\end{prop}
\begin{proof}
The  outline of the computation required to derive the above expression  has already been given in lemma \ref{le:1f}. In particular writing ${{\hat{\nabla}}}_Y=\nabla_Y+\Sigma_Y$ and using (\ref{pbiforms}), one has  that
\begin{eqnarray}
\nabla_Y\chi_p
&=&-i^{[{p\over2}]}
\langle\epsilon, (\Sigma_Y^\dagger\cdot \wedge^p\Gamma+\wedge^p\Gamma\cdot \Sigma_Y)\epsilon\rangle \ .
\end{eqnarray}
The right-hand-side of the expression above can be re-expressed in terms of forms $\chi_p$ of the hierarchy and $F$ alter some more Clifford computations.
Indeed this can be re-arranged as
\begin{eqnarray}
&&\nabla_Y \chi_p+a_{p,p+1}(\bar k_2-4\bar k_1+(-1)^p (k_2-4k_1)) i_{i_Y F} \chi_{p+1}
\cr &&
~~~~+a_{p,p-1}(\bar k_2-4\bar k_1+(-1)^{p-1} (k_2-4k_1)) i_Y F\wedge \chi_{p-1}=
\cr
&&~~~~~2 a_{p,p+3}(\bar k_1+(-1)^{p-1} k_1)i_Y(F\vee \chi_{p+3})
\cr
&&~~~~+2a_{p,p+1} (\bar k_1+(-1)^{p} k_1) i_Y(i_F\chi_{p+1})
\cr &&
~~~~-2a_{p,p-1} (\bar k_1+(-1)^{p-1} k_1) i_Y(F\wedge \chi_{p-1})
\cr &&
~~~~+2a_{p,p+1}(\bar k_1+(-1)^{p} k_1)\alpha_Y\wedge (F\vee \chi_{p+1})
\cr &&
~~~~+2a_{p,p-1} (\bar k_1+(-1)^{p-1} k_1)
\alpha_Y\wedge i_F\chi_{p-1}
\cr
&&~~~~-2a_{p,p-3} (\bar k_1+(-1)^{p} k_1) \alpha_Y\wedge F\wedge \chi_{p-3}~.
\label{bi2frep}
\end{eqnarray}
The proof proceeds by noticing that the right-hand-side of the above expression can be entirely  determined from the skew-symmetric and trace representations of the left-hand side.
Replacing the right-hand-side of (\ref{bi2frep}) with these two representations, one derives (\ref{fh2f}).
\end{proof}

{\bf Remark:}  The covariant form hierarchy untwists provided that $k_2=4k_1$.  In such a case, the condition (\ref{fh2f})  reduces to that of the conformal Killing-Yano forms. Nevertheless notice that for $k_2=4k_1$ both ${{\hat{\nabla}}}$ and $\slashed\hn$ exhibit non-trivial terms that depend on $F$.  Therefore $k_2=4k_1$ is a special region in parameter space for the hierarchy.   This region is complementary to the region of parameters that the
theorem \ref{th:2fc} applies. Indeed for $k_2=4k_1$, one has that ${\rm Im}\,e=0$.  Thus $\cos^2\psi=1$ and the condition $c_2\geq0$ requires that $n\leq 1$.  So the simplification of the hierarchy for $k_2=4k_1$ is not related to the estimates in theorem \ref{th:2fc}.

{\bf Remark:} As in the previous case of 1-form modified Dirac operator (\ref{bi2frep}) implies (\ref{fh2f}) but   the converse statement does not always hold. An argument for this can again  be established in the range $k_2=4k_1$ of parameters  where (\ref{bi2frep}) retains the dependence on $F$ but (\ref{fh2f}) does not. Therefore (\ref{fh2f}) does not imply  (\ref{bi2frep}) without some additional input.

\section{3-form modified Dirac operators }

Let $H$ be a real 3-form on $M^n$. Then set  ${\cal{D}} = {\slashed{\nabla}}+e\,{\slashed{H}}$ and
\begin{eqnarray}
{{\hat{\nabla}}}_X = \nabla_X + k_1 {\slashed{H}}\cdot X  + k_2 i_X\slashed H~,
\end{eqnarray}
where  $e, k_1, k_2\in \mathbb{C}$. The 3-form Dirac operator has the property that ${\cal{D}}: \Gamma(S^\pm)\rightarrow \Gamma(S^\mp)$, where $S^\pm$ are the two
chiral spin bundles on a spin manifold with even dimension. Moreover ${\cal{D}}$ is formally anti-self-adjoint provided that $e$ is real.

\begin{prop}
The fundamental identity for 3-form Dirac operators is
\begin{eqnarray}
\nabla^2 \parallel \epsilon \parallel^2 &=& 2 \parallel {{\hat{\nabla}}} \epsilon\parallel^2
+\left({1 \over 2} R +c_1 H^2\right)\parallel \epsilon \parallel^2
+  c_2 \parallel \slashed H\epsilon\parallel^2
\nonumber \\
&+&2 {\rm Re} \langle \epsilon, \big({\slashed{\nabla}}+(-2\bar k_1+e)\slashed H\big) {\cal{D}} \epsilon \rangle
-{1 \over 2} {\rm Re} \langle \epsilon, e\,
{\slashed{dH}} \epsilon \rangle+6 {\rm Re} \langle \epsilon, e\,
\delta\slashed H \epsilon \rangle
\cr
&+& {\rm Re} \langle \epsilon, \bigg((12(2 {\bar{k}}_1-e)+4{\bar{k}}_2)H^i{}_{pq}\Gamma^{pq} \bigg) \nabla_i \epsilon \rangle
\label{max3}
\end{eqnarray}
where
\begin{eqnarray}
c_1&=&-96 |k_1|^2-{8\over3}|k_2|^2-32 {\rm Re} (\bar k_2 k_1)~,
\cr
c_2&=&-{1\over9}\bigg(18 (n-8) |k_1|^2+2 |k_2|^2 -12 {\rm Re} (\bar k_2 k_1)\bigg)   +{\rm Re}((-4\bar k_1+2e)e)
\nonumber \\
\label{cc3x}
\end{eqnarray}
\end{prop}

\begin{proof}
It follows from a straightforward computation similar to that already presented  for $k$-form Dirac operators, $k=0,1,2$.
\end{proof}

\begin{theorem}
Let $M^n$ be a closed spin manifold.
Suppose that $H$ is harmonic, $H\not=0$,  and $k_2=-6k_1+3\bar e$. It follows that
\begin{enumerate}

\item if ${1\over 2 } R-24|e|^2 H^2\gneq 0$ and $c_2\geq 0$, then ${\rm Ker}\,{\cal{D}}=\{0\}$

\item  if ${1\over 2 } R-24|e|^2 H^2= 0$,  $c_2= 0$, and ${\rm Ker}\,{\cal{D}}\not=\{0\}$, then ${\rm Ker}\,{\cal{D}}={\rm Ker}\,{{\hat{\nabla}}}$,  $e=(6-n) k_1+k_2$ and  $n\geq 8$.
\end{enumerate}
\end{theorem}
\begin{proof}
To prove the first part of the theorem observe that under the assumptions that have been made, the last three terms of the fundamental identity (\ref{max3}) vanish.
Then assuming that $\epsilon\in {\rm Ker}\,{\cal{D}}$ and integrating over $M^n$, one arrives at a contradiction provided that $c_2\geq 0$. It remains to show that
there is a range of parameters that $c_2\geq 0$.  Indeed substituting $k_2=-6k_1+3\bar e$ into the expression for $c_2$, one finds that
\begin{eqnarray}
c_2=- (2n|k_1|^2-12 {\rm Re}(e\,k_1)+4{\rm Re}(e\,\bar k_1)-2 {\rm Re}(e^2)+ 2|e|^2)~.
\end{eqnarray}
This can be re-arranged as
\begin{eqnarray}
c_2&=&-2n\big[ \big({\rm Re} \ k_1 -{2 |e| \over n} \cos \psi \big)^2+ \big({\rm Im} \ k_1 +{4|e| \over n}\sin \psi \big)^2
\nonumber \\
&+&{2|e|^2 \over n^2} \big( n \sin^2 \psi -2 -6 \sin^2 \psi  \big)\big]~,
\label{cc3xy}
\end{eqnarray}
where $e=|e| \exp(i\psi)$.
For $c_2\geq 0$, it is required that \begin{eqnarray}
\label{threeb1}
n \sin^2 \psi -2 -6 \sin^2 \psi  \leq 0~.
\end{eqnarray}
This always holds either for $n=6$ or for $\sin\psi=0$. Otherwise the parameter $e$ has to be restricted as
\begin{eqnarray}
n \leq 6+{2 \over \sin^2 \psi}~.
\end{eqnarray}
Note also that $c_1=-24 |e|^2$ which can be derived after substituting $k_2=-6k_1+3\bar e$ into (\ref{cc3x}).

To prove the second part of the statement,  integrate the fundamental identity (\ref{max3}) over  $M^n$.  After  using the assumptions of the theorem as well as  $\epsilon\in{\rm Ker}\,{\cal{D}}$ one finds that $\epsilon\in{\rm Ker}\,{{\hat{\nabla}}}$.  So ${\rm Ker}\,{\cal{D}}\subseteq {\rm Ker}\,{{\hat{\nabla}}}$. In fact ${\rm Ker}\,{\cal{D}}= {\rm Ker}\,{{\hat{\nabla}}}$.  To see this observe that $({\cal{D}}-\slashed\hn)\epsilon=0$ and so $(e-(6-n) k_1-k_2)\slashed H=0$.  If $e=(6-n) k_1+k_2$, one has that ${\cal{D}}=\slashed\hn$ and so ${\rm Ker}\,{\cal{D}}= {\rm Ker}\,{{\hat{\nabla}}}$ as in general ${\rm Ker}\,{{\hat{\nabla}}}\subseteq {\rm Ker}\,\slashed\hn$.  Otherwise take $e\not=(6-n) k_1+k_2$ and so $\slashed H\epsilon=0$.  To continue, one can establish the identity
\begin{eqnarray}
\slashed \nabla (\slashed H \epsilon)&=&{1\over 4} \slashed dH\epsilon- 3\delta\slashed H \epsilon+ \Gamma^i\slashed H {{\hat{\nabla}}}_i \epsilon
\cr
&&
+
({k_2\over3}-(n-8) k_1) \slashed H^2\epsilon+ 8 (6k_1+k_2) H^2\epsilon~.
\label{dhe3}
\end{eqnarray}
Imposing the conditions that $H$ is harmonic, $\slashed H\epsilon=0$ and ${{\hat{\nabla}}}\epsilon=0$, the above identity gives that $(6k_1+k_2) H^2\epsilon=3\bar e H^2\epsilon=0$.  As $\epsilon$ is nowhere vanishing because it is parallel and for $e\not=0$, one concludes that $H=0$ which is a contradiction.

It remains to show that there is a range of parameters such that the second part of the theorem is valid. Indeed, the condition $e=(6-n) k_1+k_2$ substituted in
(\ref{cc3xy}) gives
\begin{eqnarray}
c_2=-{4|e|^2 \over n} \big( n \sin^2 \psi -2 -6 \sin^2 \psi  \big)~.
\end{eqnarray}
So the requirement that  $c_2=0$ gives
\begin{eqnarray}
 (n-6) \sin^2 \psi =2~.
 \end{eqnarray}
This has solutions for all $n\geq 8$.
\end{proof}

\begin{prop}
Let $M^n$ be a closed spin manifold.
Let $\lambda$ be an eigenvalue of ${\cal{D}}$.
If $H$ is harmonic,   $k_1=-i {\rm Im}\, e$ and   $k_2=6i {\rm Im}\, e+3\bar e$, then \begin{eqnarray}
|\lambda|^2\geq \inf_{M^n} \left({1 \over 4}R -12|e|^2 H^2\right)
\end{eqnarray}
provided that either ${\rm Im}\, e=0$ or $n\leq 6$.
\end{prop}
\begin{proof}

Under the assumptions made on the parameters,
the fundamental identity (\ref{max3}) can be rearranged as
\begin{eqnarray}
\nabla^2 \parallel \epsilon \parallel^2 &=& 2 \parallel {{\hat{\nabla}}}\epsilon\parallel^2
+ \big({1 \over 2}R -24 |e|^2 H^2 \big) \parallel \epsilon \parallel^2-2 \parallel{\cal{D}}\epsilon\parallel^2
\nonumber \\
&+&2\nabla_i{\rm Re} \langle \epsilon, \Gamma^i {\cal{D}} \epsilon \rangle
+c_2 \parallel {\slashed{H}} \epsilon \parallel^2~.
\label{3max3}
\end{eqnarray}
Integrating the above expression over $M^n$ gives a bound on the eigenvalues of ${\cal{D}}$ provided that $c_2\geq 0$.  A direct computation reveals that
\begin{eqnarray}
c_2=-2(n-6) |e|^2 \sin^2\psi~.
\end{eqnarray}
So $c_2\geq 0$ provided that either $n\leq 6$ or ${\rm Im}\,e=0$. In the latter case $k_1=0$.
\end{proof}

\begin{theorem}\label{th:3f}
Let $M^n$ be a closed spin manifold.
Suppose that $H$ is closed 3-form $dH=0$, $H\not=0$,  and $k_2=-6k_1+3{\rm Re}\,e$.
\begin{enumerate}

\item In addition  if ${1 \over 2}R -24 ({\rm Re}\,e)^2 H^2\gneq 0$ and $c_2\geq 0$, then ${\rm Ker}\,{\cal{D}}=\{0\}$.

\item Moreover if ${1 \over 2}R -24 ({\rm Re}\,e)^2 H^2= 0$, $c_2= 0$, and ${\rm Ker}\,{\cal{D}}\not=\{0\}$, then $e=(6-n)k_1+k_2$ and ${\rm Ker}\,{\cal{D}}={\rm Ker}\,{{\hat{\nabla}}}$~.

\end{enumerate}

\end{theorem}
\begin{proof}
If $dH=0$ and  $k_2=-6k_1+3{\rm Re}\,e$, the fundamental identity (\ref{max3}) can be rearranged as

\begin{eqnarray}
\nabla^2 \parallel \epsilon \parallel^2 &=& 2 \parallel {{\hat{\nabla}}} \epsilon, \parallel^2
+ \big({1 \over 2}R -24 ({\rm Re}\,e)^2 H^2 \big) \parallel \epsilon \parallel^2-6\nabla^i{\rm Re}\langle\epsilon, e H_{ijk}\Gamma^{jk}\epsilon\rangle
\cr
&+& 2 {\rm Re} \langle \epsilon, \big({\slashed{\nabla}}-(2{\bar{k}}_1-e){\slashed{H}} \big) {\cal{D}} \epsilon \rangle +c_2 \parallel {\slashed{H}} \epsilon \parallel^2~,
\label{5max3}
\end{eqnarray}
where
\begin{eqnarray}
c_2=-\left(2n |k_1|^2-12 {\rm Re}\, k_1 {\rm Re}\, e+2({\rm Re}\, e)^2+4 {\rm Re}(\bar k_1e)-2 {\rm Re}( e^2)\right)~.
\end{eqnarray}

It is clear that integrating (\ref{5max3}) over $M^n$ and assuming that $\epsilon\in {\rm Ker}\,{\cal{D}}$ leads to a contradiction provided that $c_2\geq 0$.
Thus ${\rm Ker}\,{\cal{D}}=\{0\}$ which proves the first part of the statement.  It remains to show that there is a range of parameters such that $c_2\geq 0$.  Indeed
$c_2$ can be rewritten as
\begin{eqnarray}
c_2&=&-2n \bigg( ({\rm Re}\,k_1-{2\over n}{\rm Re}\,e)^2+({\rm Im}\,k_1+{1\over n}{\rm Im}\,e)^2
\nonumber \\
&-&{4|e|^2\over n^2}+{3+n\over n^2}|e|^2\sin^2\psi\bigg)~.
\end{eqnarray}
For $c_2\geq 0$, one has to set
\begin{eqnarray}
\sin^2\psi\leq {4\over n+3}~,
\end{eqnarray}
which can always be satisfied for any $n$.

To prove the second part of the theorem,  integrate the fundamental identity (\ref{5max3}) over  $M^n$ and after  using the assumptions of the theorem as well as  $\epsilon\in{\rm Ker}\,{\cal{D}}$ one finds that $\epsilon\in{\rm Ker}\,{{\hat{\nabla}}}$.  So one has ${\rm Ker}\,{\cal{D}}\subseteq {\rm Ker}\,{{\hat{\nabla}}}$.

In fact ${\rm Ker}\,{\cal{D}}= {\rm Ker}\,{{\hat{\nabla}}}$.  To see this observe that $({\cal{D}}-\slashed\hn)\epsilon=0$ and so $(e-(6-n) k_1-k_2)\slashed H=0$.  If $e=(6-n) k_1+k_2$, observe that ${\cal{D}}=\slashed\hn$ and so ${\rm Ker}\,{\cal{D}}= {\rm Ker}\,{{\hat{\nabla}}}$ as in general ${\rm Ker}\,{{\hat{\nabla}}}\subseteq {\rm Ker}\,\slashed\hn$.  Otherwise take $e\not=(6-n) k_1+k_2$ and so $\slashed H\epsilon=0$.  Then
the condition ({\ref{dhe3}}) implies that
\begin{eqnarray}
-3\delta\slashed H \epsilon +8(k_2+6k_1) H^2 \epsilon=0 \ .
\end{eqnarray}
On taking the real part of the Dirac inner product of this expression with $\epsilon$, and using
the condition $k_2=-6k_1+3{\rm Re}\,e$, one finds $H^2 \parallel \epsilon \parallel^2 {\rm Re}\,e=0$. Assuming that $H$ and $\epsilon$ do not vanish identically, this implies that
${\rm Re}\,e=0$. Using this condition, the expression for $c_2$ can be rewritten as
\begin{eqnarray}
c_2 = -2n \bigg(  ({\rm Re} \, k_1)^2 +({\rm Im} \, k_1 +{1 \over n} {\rm Im} \, e)^2
+{|e|^2 \over n^2}(n-1) \bigg)
\end{eqnarray}
and hence the condition $c_2=0$ implies that $k_1=k_2=e=0$, and we discard this case as it implies that the 3-form does not modify either the Dirac operator or the supercovariant derivative.

It remains to demonstrate that there is a range of parameters such that $c_2=0$.  Indeed $e=(6-n) k_1+k_2$ implies that
\begin{eqnarray}
c_2=-2n |e|^2\left( -{4\over n^2}+{3+n\over n^2}\sin^2\psi\right)~.
\end{eqnarray}
So $c_2=0$ gives $\sin^2\psi= {4\over n+3}$ and an $e$ can be chosen for any $n$.
\end{proof}

{\bf Remark:} As the conditions of the above theorem require  $ R-48({\rm Re}\,e)^2 H^2\geq  0$, one has that $R\gneq 0$.  Therefore
 the real index of the Dirac operator viewed a ${\rm Cl}_n$-linear operator vanishes-this index is also refereed to as $\alpha$-invariant, see e.g. \cite{lawson}.
  Therefore all manifolds that satisfy the conditions of theorem  \ref{th:3f} have vanishing $\alpha$-invariant.   As has been already mentioned ${\cal{D}}: \Gamma(S^\pm)\rightarrow \Gamma(S^\mp)$ and so ${\cal{D}}$ can extend to a
${\rm Cl}_n$-linear operator on the Clifford bundle such that ${\cal{D}}: \Gamma({\rm Cl}_n(M))^{\rm ev, od}\rightarrow \Gamma({\rm Cl}_n(M))^{\rm od, ev}$.
For $e$ real ${\cal{D}}$ is formally anti-self adjoint and has the same principal symbol as the Dirac operator, the real index of ${\cal{D}}$ as a ${\rm Cl}_n$-linear operator is the
same as the index of the Dirac operator. Similar comments can be made for 1-form modified Dirac operators.

\begin{prop}
\label{prop3}
Let $M^n$ be a closed spin manifold.
If $dH=0$, $c_2\geq0$,  $k_1=-i{\rm Im}\, e$ and $k_2=6 i{\rm Im}\, e +3{\rm Re}\,e$, then the eigenvalues $\lambda$ of ${\cal{D}}$ are bounded as
\begin{eqnarray}
|\lambda|^2\geq \inf_{M^n} \left({1 \over 4}R -12 ({\rm Re}\,e)^2 H^2\right)~,
\end{eqnarray}
provided that ${\rm Im}\, e=0$ and so ${\cal{D}}$ is a formally anti-self-adjoint operator.
\end{prop}

\begin{proof}
Under the assumptions on the parameters, the fundamental identity (\ref{max3}) can be re-arranged as
\begin{eqnarray}
\nabla^2 \parallel \epsilon \parallel^2 &=& 2 \parallel {{\hat{\nabla}}} \epsilon\parallel^2-2 \parallel{\cal{D}}\epsilon\parallel^2
+ \big({1 \over 2}R -24 ({\rm Re}\,e)^2 H^2 \big) \parallel \epsilon \parallel^2
\nonumber \\
&-&6\nabla^i{\rm Re}\langle\epsilon, eH_{ijk}\Gamma^{jk}\epsilon\rangle
+ 2 \nabla_i{\rm Re} \langle \epsilon, \Gamma^i {\cal{D}} \epsilon \rangle +c_2 \parallel {\slashed{H}} \epsilon \parallel^2
\label{6max3v}
\end{eqnarray}
where now
\begin{eqnarray}
c_2=-2 (n-1)\sin^2\psi~.
\end{eqnarray}
The proposition follows after assuming that $\epsilon$ is an eigenspinor of ${\cal{D}}$ with eigenvalue $\lambda$ and integrating (\ref{6max3v}) over $M^n$.
The requirement that $c_2\geq0$ for $n>1$ leads to the condition that ${\rm Im}\,e=0$.  Therefore $k_1=0$.  ${\cal{D}}$ is a formally anti-self-adjoint
operator and the holonomy of ${{\hat{\nabla}}}$ is contained in ${\rm Spin}(n)$.
\end{proof}

\subsection{Holonomy of  ${{\hat{\nabla}}}$ connections and covariant form hierarchies}

It is straightforward to observe that ${{\hat{\nabla}}}_Y: \Gamma(S^\pm)\rightarrow \Gamma(S^\pm)$, where $S^\pm$ are the two chiral spin bundles on even dimensional manifolds $M^n$. Therefore ${{\hat{\nabla}}}_Y$ preserves the chirality of sections of the spin bundle and so it is expected that the Lie algebra of the holonomy of ${{\hat{\nabla}}}$ to be restricted on the even part ${\rm Cl}^{\rm ev}_n$ of the Clifford algebra ${\rm Cl}_n$.

\begin{theorem}
The Lie algebra of the holonomy of ${{\hat{\nabla}}}$, $\mathfrak{Lie}\, {\rm Hol}({{\hat{\nabla}}})$,  for a generic choice of metric on $M^n$, $H$ and $k_1,k_2$ is $\mathfrak{Lie}\, {\rm Hol}({{\hat{\nabla}}})\subseteq (\rm{Cl}^{\rm ev}-\rm{Cl}^0)\otimes\mathbb{C}$.  Furthermore if $k_1=0$, then $\mathfrak{Lie}\, {\rm Hol}({{\hat{\nabla}}})\subseteq \mathfrak{spin}(n)\otimes \mathbb{C}$.
\end{theorem}
\begin{proof}
To use the Ambrose-Singer theorem to identify the Lie algebra of the holonomy group of ${{\hat{\nabla}}}$, it is sufficient to compute the curvature of ${{\hat{\nabla}}}$.  Indeed after some computation, one finds that
\begin{eqnarray}
\hat R(X,Y)&=&R(X,Y)+k_1 (\nabla_X \slashed H \cdot \slashed Y-\nabla_Y \slashed H \cdot \slashed X)+ k_2 (\nabla_X i_Y \slashed H-\nabla_Y i_X \slashed H)
\cr &&
~~~~+k_1^2 (\slashed H\cdot \slashed X\cdot \slashed H\cdot \slashed Y-\slashed H\cdot \slashed Y\cdot \slashed H\cdot \slashed X)
\cr
&&
~~~~
+k_2^2 (i_X\slashed H \cdot i_Y\slashed H-i_Y\slashed H \cdot i_X\slashed H)+k_1k_2( \slashed H \cdot \slashed X \cdot i_Y\slashed H
\cr
&&~~~~+ i_X\slashed H \cdot \slashed H \cdot \slashed Y
-\slashed H \cdot \slashed Y \cdot i_X\slashed H- i_Y\slashed H \cdot \slashed H \cdot \slashed X)~.
\end{eqnarray}
It is straightforward to observe that $\hat R(X,Y)$ at every point on $M^n$ is an element of $(\rm{Cl}^{\rm even}-\rm{Cl}^0)\otimes \mathbb{C}$.  Furthermore if $k_1=0$, then
$\hat R(X,Y)$ is an element of $\rm{Cl}^2\otimes\mathbb{C}$ which is identified with $\mathfrak{spin}(n)\otimes \mathbb{C}$.
\end{proof}

\begin{prop}\label{p:3fb}
The twisted covariant form hierarchy associated with geometry of $M^n$ that admits a ${{\hat{\nabla}}}$-parallel spinor is described by the equation
\begin{eqnarray}
 \nabla_Y \chi_p+4 {\rm Im}\,(k_2+6k_1) {i_YH}\vee \chi_{p+2}
\nonumber \\
-4{\rm Re}\, (k_2+6k_1) i_{i_Y H}\chi_p
+4 {\rm Im}\,(k_2+6k_1)\,i_Y H\wedge \chi_{p-2}
\nonumber \\
={1\over p+1} i_Y\bigg( d\chi_p
+4 \, {\rm Im} (k_2+6k_1)\,  i^\dagger_H \chi_{p+2}
\nonumber \\
+8\,{\rm Re}\,  (k_2+6k_1)\,
i_H\chi_p+12 \,{\rm Im}\,(k_2+6k_1)\, H\wedge \chi_{p-2}\bigg)
\nonumber \\
~~~~-{1\over n-p+1} \alpha_Y\wedge \bigg(\delta \chi_p
-12\,{\rm Im} (k_2+6k_1) H\vee \chi_{p+2}
\nonumber \\
~~~~+8\,{\rm Re} (k_2+6k_1) i^\dagger_H \chi_p-4\,{\rm Im}(k_2+6k_1)\, i_H\chi_{p-2} \bigg)~.
\label{cov4f}
\end{eqnarray}
\end{prop}
\begin{proof}
The computation required to derive the formula above has already been  described for the 1-form and 2-form modified Dirac operators.  In particular
after some Clifford algebra,  one finds that
\begin{eqnarray}
&&\nabla_Y \chi_p+ 4\,{\rm Im} (k_2+6k_1) i_YH\vee \chi_{p+2}-4\,{\rm Re}(k_2+6k_1) i_{i_Y H} \chi_p
\cr
&&\qquad\qquad + 4\,{\rm Im}(k_2+6k_1) i_Y H\wedge \chi_{p-2}=
\cr &&
\qquad 12\, {\rm Re} k_1 i_Y(H\vee \chi_{p+4})+12\, {\rm Im}\,k_1\,i_Y(i^\dagger_H\chi_{p+2})
\cr
&& \qquad +12\, {\rm Re}\,k_1\, i_Y(i_H \chi_p) +12\, {\rm Im}\, k_1\, i_Y(H\wedge \chi_{p-2})
\cr
&& \qquad +12\, {\rm Im}\,k_1\, \alpha_Y\wedge(H\vee \chi_{p+2})-12\, {\rm Re}\,k_1\, \alpha_Y\wedge i^\dagger_H\chi_p
\cr
&& \qquad
+12\, {\rm Im}\,k_1\, \alpha_Y\wedge
i_H\chi_{p-2}-12\, {\rm Re}\,k_1\, \alpha_Y\wedge H\wedge \chi_{p-4}~.
\label{xxcov3}
\end{eqnarray}
Then (\ref{cov4f}) follows from (\ref{xxcov3}) because the right-hand-side of the above equation is specified from the skew and trace representations
of the left-hand-side of the same expression. So after replacing the right-hand-side of (\ref{xxcov3}) with the skew and trace representations of the
left-hand-side, one arrives at (\ref{cov4f}) which proves the proposition.
\end{proof}

{\bf Remark:} The equation  (\ref{cov4f}) of the covariant form hierarchy for $k_2=-6k_1$ simplifies to the Killing-Yano form equation.
Again both ${{\hat{\nabla}}}$ and $\slashed\hn$ exhibit non-trivial dependence on $H$ in this region of parameter space. The region $k_2=-6k_1$ is again complementary
to the region of parameters for which the theorem  \ref{th:3f} is valid.  Therefore the simplification of the hierarchy is not due to the bounds described
in  theorem  \ref{th:3f}.

\subsection{$k$-form, $k>3$,  modified Dirac operators}

Here we shall demonstrate that the results we have proven so far do not generalize to 4-form modified Dirac operators. Indeed
consider the 4-form modified Dirac operator ${\cal{D}}=\slashed \nabla+ e \slashed F$ and
\begin{eqnarray}
{{\hat{\nabla}}}_X=\nabla_X+k_1 \slashed F\cdot \slashed X+ k_2 i_X \slashed F~,
\end{eqnarray}
where  $F$ is a real 4-form on $M^n$ and $e, k_1, k_2\in \mathbb{C}$.  Below we state without proof the fundamental identity.

\begin{prop}
The fundamental identity of 4-form modified Dirac operators is
\begin{eqnarray}
\nabla^2\parallel\epsilon\parallel^2&=&2 \parallel{{\hat{\nabla}}}\epsilon\parallel^2+{1\over 2} R \parallel \epsilon\parallel^2+2{\rm Re}\langle\epsilon, (\slashed\nabla-(2\bar k_1+e)\slashed F) {\cal{D}} \epsilon\rangle
\cr
&&
+c_1  \parallel \slashed F\parallel^2-4{\rm Re}\langle \epsilon, (8\bar k_1-\bar k_2+4{\rm Re} \,e) F^i{}_{j_1j_2j_3} \Gamma^{j_1j_2j_3} \nabla_i \epsilon\rangle
\cr
&&
+18 c_2 \langle  F^{mn}{}_{i_1i_2}\Gamma^{i_1i_2}\epsilon,  F_{mni_3i_4}\Gamma^{i_3i_4} \epsilon\rangle-24  c_2 F^2 \parallel\epsilon\parallel^2
\cr
&&
-8{\rm Re}\nabla_i\langle\epsilon, e F^i{}_{jkl} \Gamma^{jkl}\epsilon\rangle-{2\over5}{\rm Re}\langle\epsilon, e \slashed{dF}\epsilon\rangle~.
\label{max4}
\end{eqnarray}
where
\begin{eqnarray}
c_1&=&{\rm Re}\left(4\bar k_1 e + 2 e^2-2 (n-16) |k_1|^2- 4\bar k_2 k_1\right)
\nonumber \\
c_2&=&-{\rm Re}\,\left(64 |k_1|^2-16 \bar k_2 k_1+|k_2|^2\right) \ .
\end{eqnarray}
\end{prop}

{\bf Remark:} Suppose that $M^n$ be a closed spin manifold, and
let $F$ be a closed 4-form $dF=0$.  Assuming that
 ${1\over2} R-24  c_2 F^2\gneq 0$,  $k_2=8k_1+4{\rm Re}\, e$ and $c_1, c_2\geq 0$, one could conclude that ${\rm Ker}\, {\cal{D}}=\{0\}$.  However
 there is no range of parameters for which this holds.  Indeed, the $\nabla\epsilon$ term in (\ref{max4}) vanishes provided that
   $k_2=8k_1+4\bar e$.  Using this, one finds that
   \begin{eqnarray}
   c_2=-16 ({\rm Re}\, e)^2~.
   \end{eqnarray}
As   $c_2\geq 0$, this requires that  ${\rm Re}\, e=0$.  In turn a short calculation reveals that
\begin{eqnarray}
c_1=-2 ({\rm Im} e-{\rm Im} k_1)^2-2 (n-1) ({\rm Im} k_1)^2-2n ({\rm Re} k_1)^2~.
\end{eqnarray}
As it is required that $c_1\geq 0$, one has to set $k_1=e=0$ and so the Dirac operator does not get modified.
A similar conclusion is expected to hold for all generically  $k$-form modified Dirac operators for $k\geq 4$. This does not rule out the
possibility that there are may be counter-examples to this for specially chosen $k$-forms, $k\geq 4$, e.g (anti-)-self dual forms,  representations for the
 spinor $\epsilon$ and dimension of $M^n$.  However generically, it will not be possible to proceed for reason similar to those exhibited for $k=4$ above.

\section{(0,1)-multi-form modified Dirac operators}

As estimates for the eigenvalues of $k$-forms, $k>3$,  modified Dirac operators cannot  generically be obtained, we shall turn our attention
to multi-form modified Dirac operators.  As a first example consider
the (0,1)-multi-form modified Dirac operator  which can be written as
\begin{eqnarray}
{\cal{D}}=\slashed{\nabla}+ e_1 f + e_2 \slashed A~,
\end{eqnarray}
where $e_1, e_2\in \mathbb{C}$,  and $f$ is a function and $A$ is a 1-form on $M^n$.  Furthermore consider
the connection
\begin{eqnarray}
{{\hat{\nabla}}}_X=\nabla_X+k_0 f \slashed X +k_1 \slashed A \cdot \slashed X+k_2 A(X)
\end{eqnarray}
on the spin bundle, where $k_0,  k_1, k_2\in \mathbb{C}$.

\begin{prop}
The fundamental identity for 0- and 1-form Dirac operator ${\cal{D}}$ is
\begin{eqnarray}
\nabla^2\parallel\epsilon\parallel^2&=&2\parallel{{\hat{\nabla}}}\epsilon\parallel^2+{1\over2} R \parallel\epsilon\parallel^2 +(c_0 f^2+c_1 A^2)\parallel\epsilon\parallel^2
\cr
&&+2{\rm Re}\,\langle\epsilon,
\left(\slashed{\nabla}-2\bar k_0 f+\bar e_1 f+(2\bar k_1+e_2)\slashed{A}\right){\cal{D}}\epsilon\rangle
\cr
&&-2{\rm Re} (2k_1+k_2) \delta A \parallel\epsilon\parallel^2
-2\nabla_i{\rm Re}\,\langle\epsilon, \Gamma^i e_1 f \epsilon\rangle
\cr
&&+2{\rm Re} (2k_1+k_2+e_2)\delta\left( A\parallel\epsilon\parallel^2\right)
+{\rm Re}\,\langle\epsilon, c_2 f \slashed{A}\epsilon\rangle
\cr
&&+4 {\rm Im} (e_2+2\bar k_1+\bar k_2)\, {\rm Im}\langle\epsilon, \nabla_A\epsilon\rangle
-{\rm Re}\, \langle\epsilon, e_2 \slashed{dA}\epsilon\rangle~,
\label{max01}
\end{eqnarray}
where
\begin{eqnarray}
c_0&=&4{\rm Re}\,(\bar k_0  e_1)+2{\rm Re}\,e_1^2 -2n|k_0 |^2-4({\rm Re}\, e_1)^2~,
\cr
c_1&=& -(4{\rm Re}\,(\bar k_1 e_2)+2n|k_1|^2+4{\rm Re}\,(\bar k_1 k_2)
+2|k_2|^2
+2{\rm Re}\,e_2^2)~,
\cr
c_2&=&4(\bar k_0 -{\rm Re}\, e_1 ) e_2-4\bar k_1 e_1 -2(2-n) \bar k_0  k_1
\nonumber \\
&-&2(2-n) k_0  \bar k_1-2\bar k_0  k_2-2 \bar k_2 k_0~.
\label{acon01}
\end{eqnarray}
\end{prop}

\begin{proof}
The derivation of the formula is similar to that of the fundamental identities  of 0-form Dirac operators and 1-form Dirac operators. One of the essential
new terms is that with coefficient $c_2$.  As  will be demonstrated below, this has to vanish to construct bounds for the eigenvalues of ${\cal{D}}$.
\end{proof}

\begin{theorem}
\label{theorem01}
Let $M^n$ be a closed spin manifold.
Let $A$ be closed 1-form $dA=0$, ${\rm Im} (e_2+2\bar k_1+\bar k_2)=0$ and ${\rm Re}\,c_2=0$.
\begin{enumerate}
\item If ${1\over2} R +c_0 f^2+c_1 A^2-2{\rm Re} (2k_1+k_2)\delta A\gneq 0$, then ${\rm Ker}{\cal{D}}=\{0\}$.

\item If ${1\over2} R +c_0 f^2+c_1 A^2-2{\rm Re} (2k_1+k_2)\delta A=0$, and $e_1=nk_0$,
$A$ is non-vanishing at some point on $M^n$, and
${\rm Ker}{\cal{D}} \neq \{ 0 \}$, then
$e_2=(2-n)k_1+k_2$ and ${\rm Ker}\,{\cal{D}}={\rm Ker}\,{{\hat{\nabla}}}$.
\end{enumerate}
\end{theorem}

\begin{proof}
To prove the first part of the statement observe that as $A$ is closed, ${\rm Im} (e_2+2\bar k_1+\bar k_2)=0$  and ${\rm Re}\, c_2=0$, the last three terms of the fundamental identity vanish. Assuming that ${\cal{D}}\epsilon=0$ and integrating the rest of the identity over $M^n$  using the assumptions of the theorem, one is led to a contradiction. Therefore ${\rm Ker}\, {\cal{D}}=\{0\}$.
The conditions ${\rm Im} (e_2+2\bar k_1+\bar k_2)=0$ and ${\rm Re}\,c_2=0$  on the parameters have solutions. The statement is valid for an 8-parameter family.

The second part of the statement can be proved in a similar way.  In particular integrating the fundamental identity, one concludes that ${\rm Ker} \, {\cal{D}}\subseteq {\rm Ker}{{\hat{\nabla}}}$.
In turn, this implies that if $\epsilon \in {\rm Ker} \, {\cal{D}}$
then $\slashed\hn\epsilon=0$ and so ${\cal{D}}\epsilon-\slashed\hn\epsilon=(e_2-(2-n) k_1-k_2)\slashed A\epsilon=0$,
on using the condition $e_1=nk_0$.
Therefore $(e_2-(2-n) k_1-k_2) A^2\epsilon=0$. As $\epsilon$  is no-where zero on $M^n$ because it is parallel and    as $A\not=0$,
one concludes that $e_2=(2-n) k_1+k_2$. From this it follows that ${\cal{D}}=\slashed\hn$ and ${\rm Ker}\, {\cal{D}}={\rm Ker}{{\hat{\nabla}}}$ as all ${{\hat{\nabla}}}$-parallel spinors are zero modes of
$\slashed\hn$.

In addition the assumption $e_1=nk_0$ and $e_2=(2-n)k_1+k_2$ imply that ${\cal{D}}=\slashed\hn$ and so one establishes ${\rm Ker}{\cal{D}}={\rm Ker}{{\hat{\nabla}}}$.
Note that ${\rm Im} (e_2+2\bar k_1+\bar k_2)=0$  and $e_2=(2-n)k_1+k_2$ give that ${\rm Im}\, k_1=0$. There is 4-parameter family  that the second part of the theorem holds.
\end{proof}

\begin{corollary}\label{cor611}
Let $M^n$ be a closed spin manifold.
Let $\eta$ be an eigenspinor of the Dirac operator with eigenvalue $\lambda$. If ${\rm Im} (e_2+2\bar k_1+\bar k_2)=0$, $e_2={\rm Re}\, e_2$  and ${\rm Re}\,c_2= 0$, then
\begin{eqnarray}
|\lambda|^2 \geq {1\over (n s^2-2s+1)} &\inf_{M^n}\bigg({1\over 4} R-{\rm Re} (2k_1+k_2)h^{-1} \nabla^2 h
\nonumber \\
&+[{1\over2}c_1+ {\rm Re} (2k_1+k_2)] |h^{-1}dh|^2 \bigg),
\end{eqnarray}
where $s\in \mathbb{R}$ and $h$ is a real positive function on $M^n$. An upper lower bound is obtained for $s=n^{-1}$, $n>1$.
\end{corollary}

\begin{proof}
To demonstrate this statement take $f=1$,  $e_1=-\lambda$, $k_0=s e_1$ and $A=-h^{-1}dh$.  To continue set $\epsilon=h^{e_2}\eta $ and observe that
${\cal{D}}\epsilon=0$.  Using that $\lambda$ is purely imaginary  a straightforward computation  reveals that $c_0=-2 (ns^2-2s+1) |\lambda|^2$.
As ${\cal{D}}$ has a non-empty kernel and ${\rm Re}\,c_2\geq 0$, it is required that
\begin{eqnarray}
{1\over2} R -2 (ns^2-2s+1) |\lambda|^2 +c_1 A^2-2{\rm Re} (2k_1+k_2)\delta A\leq 0~.
\end{eqnarray}
The corollary follows after re-arranging the above inequality for $n>1$ and substituting $A=-h^{-1}dh$.
It remains to demonstrate that there is a range of parameters that the results holds.  This will be illustrated in an example below.
\end{proof}

\begin{theorem}
Let $M^n$ be a closed spin manifold.
Let  $\lambda$ be an eigenvalue of the Dirac operator and $n\geq 3$, then
\begin{eqnarray}
|\lambda|^2\geq {n\over 4( n-1)} \mu_1~,
\end{eqnarray}
where $\mu_1$ is the first eigenvalue of the conformally invariant (Yamabe) operator $L=-4{{n-1}\over n-2}\nabla^2+R$.
This result has originally  been  demonstrated by Hijazi  \cite{hijazi}.
\end{theorem}

\begin{proof}

To prove this from the corollary above choose $k_1$ and $k_2$ to be real numbers and set ${1\over2}c_1+ {\rm Re} (2k_1+k_2)=0$. It is straightforward to
verify that the conditions of the corollary \ref{cor611} are satisfied and one finds that
\begin{eqnarray}
|\lambda|^2\geq {n\over 4( n-1)} \inf_{M^n}\left( h^{-1} L h \right)~,
\end{eqnarray}
after setting $s=n^{-1}$.

To continue one chooses $h$ to be the eigenvector of $L$ with eigenvalue $\mu_1$.  It is known that the first eigenvalue of $L$ admits an eigenvector that does not change sign on $M^n$, see \cite{hijazi}. This establishes the theorem provided the parameters can be chosen such that ${1\over2}c_1+ {\rm Re} (2k_1+k_2)=0$. Indeed such a choice of parameters is $e_2=-k_1$, $k_1=(n-2)^{-1}$ and $k_2=(n-3) (n-2)^{-1}$.

\end{proof}

\begin{prop}
\label{prop01}
Let $M^n$ be a closed spin manifold.
If $dA=0$,  ${\rm Im} (e_2+2\bar k_1+\bar k_2)=0$, ${\rm Re}\,c_2=0$,  $k_0 = e_1$ and $k_1=-{\rm Re}\, e_2$, then the eigenvalues of ${\cal{D}}$ are bounded from below as
\begin{eqnarray}
&&|\lambda|^2\geq \inf_{M^n}\bigg( {1\over4} R-{\rm Re} (-2e_2+k_2) \delta A +{1\over2} c_0 f^2 +{1\over2} c_1
A^2\bigg)~,
\end{eqnarray}
where
\begin{eqnarray}
c_0&=&2(1-n) |e_1|^2 \ ,
\nonumber \\
c_1&=&-2\left( (n-2) ({\rm Re}\, e_2)^2-2 {\rm Re}\, e_2 {\rm Re}\, k_2+ |k_2|^2+ {\rm Re}\, e_2^2\right) \ .
\end{eqnarray}

\end{prop}

\begin{proof}

Under the assumptions of the proposition, the fundamental identity can be rewritten as
\begin{eqnarray}
\nabla^2\parallel\epsilon\parallel^2&=&2\parallel{{\hat{\nabla}}}\epsilon\parallel^2+\big({1\over2} R-2{\rm Re} (-2e_2+k_2) \delta A+c_1 f^2+c_2
A^2\big) \parallel\epsilon\parallel^2
\cr
&&-2\parallel{\cal{D}}\epsilon\parallel^2+2\nabla_i{\rm Re}\,\langle\epsilon,
\Gamma^i{\cal{D}}\epsilon\rangle
-2\nabla_i{\rm Re}\,\langle\epsilon, \Gamma^i e_1 f\epsilon\rangle~.
\label{3max01}
\end{eqnarray}
Assuming that $\epsilon$ is an eigenspinor of ${\cal{D}}$ with eigenvalue $\lambda$ and after integrating (\ref{3max01}) over $M^n$, one arrives at the bound
described in the proposition.

It remains to investigate the range of parameters that the bound is valid. First the expressions for $c_0$ and $c_1$ follow from those in (\ref{acon01})
after substituting $k_0 = e_1$ and $k_1=-{\rm Re}\, e_2$.  Similarly, ${\rm Re}\,c_2=0$ yields
 \begin{eqnarray}
  {\rm Re}\,\big(\bar e_1( e_2-k_2)\big)+(2-n) {\rm Re}\,e_1 {\rm Re}\,e_2=0~.
 \end{eqnarray}
Furthermore as $k_1$ is real, ${\rm Im}\, e_2= {\rm Im}\, k_2$. So the bound on the eigenvalues $\lambda$ holds for 4-parameter family.
\end{proof}

{\bf Remark: } One can compute the curvature of ${{\hat{\nabla}}}$ and find the Lie algebra of its holonomy group. It can be seen that $\hat R$ at each point in $M^n$
for generic fields and parameters has components in ${\rm Cl}^3_n$. Closure of the Lie algebra of the holonomy group requires that in general
$\mathfrak{Lie} {\rm Hol}({{\hat{\nabla}}})\subseteq {\rm Cl}_n\otimes\mathbb{C}$.  This will be the case for all the other ${{\hat{\nabla}}}$ connections associated with multi-form Dirac
operators.  It turns out that in most cases $\mathfrak{Lie} {\rm Hol}({{\hat{\nabla}}})$ is either a subset of ${\rm Cl}_n\otimes\mathbb{C}$ or $({\rm Cl}_n-{\rm Cl}_n^0)\otimes\mathbb{C}$.

The formula for the twisted covariant form hierarchy of ${{\hat{\nabla}}}$ can derived easily from linear superposition of those derived  for the 0-form and 1-form Dirac operators.
Because of this, the explicit formula will not be stated. The same applies for  formulae of  twisted covariant form hierarchies the remaining (0,$k$), $k>1$ cases.

\section{(0,2)-multi-form modified Dirac operators}

A (0,2)-multi-form modified Dirac operator  can be written as
\begin{eqnarray}
{\cal{D}}=\slashed{\nabla}+ e_1 f + e_2 \slashed F~,
\end{eqnarray}
where $e_1, e_2\in \mathbb{C}$,  and $f$ is a function and $F$ is a 2-form on $M^n$.  Furthermore consider
the connection
\begin{eqnarray}
{{\hat{\nabla}}}_X=\nabla_X+k_0 f \slashed X +k_1 \slashed F \cdot \slashed X+k_2 \slashed{i_XF}~,
\end{eqnarray}
on the spin bundle, where $k_0,  k_1, k_2\in \mathbb{C}$.
\begin{prop}
The fundamental identity for ${\cal{D}}$ is
\begin{eqnarray}
\nabla^2\parallel\epsilon\parallel^2&=&\left({1\over2} R+ c^{[0]}_1 f^2 +c^{[2]}_1 F^2\right)\parallel\epsilon\parallel^2+2 \parallel{{\hat{\nabla}}}\epsilon\parallel^2+c^{[2]}_2 \parallel\slashed F\epsilon\parallel^2
\cr
&+&2 {\rm Re} \langle\epsilon, (\slashed {\nabla}-2\bar k_0 f+\bar e_1 f+(2{\bar{k}}_1- e_2) {\slashed{F}}){\cal{D}}\epsilon\rangle
+{\rm Re}\, \langle\epsilon, c_3 f \slashed F\epsilon\rangle
\cr
&-& 2 \nabla_i{\rm Re}\langle \epsilon, e_1 f \Gamma^i\epsilon\rangle+{\rm Re} \langle \epsilon,
(-4 {\bar{k}}_2 +16 {\bar{k}}_1-8i {\rm Im}\, e_2) F^i{}_j \Gamma^j \nabla_i \epsilon \rangle
\cr
&-&{2\over3} {\rm Re} \langle \epsilon, e_2 \slashed{dF} \epsilon\rangle-4\nabla^j {\rm Re} \langle \epsilon, e_2 F_{ji}\Gamma^i \epsilon \rangle
\label{02max}
\end{eqnarray}
where
\begin{eqnarray}
&&c^{[0]}_1=-2 |k_0|^2 n+2{\rm Re} ((2\bar k_0+e_1) e_1)-4({\rm Re}\,e_1)^2~,
\cr
&&c^{[2]}_1=-32 |k_1|^2-2|k_2|^2+16 {\rm Re} (k_1\bar k_2)~,
\cr
&&c^{[2]}_2=-2 (n-8) |k_1|^2-4{\rm Re} (k_1\bar k_2)+{\rm Re}(e_2 (4\bar k_1-2e_2))~,
\cr
&&
c_3=4(i {\rm Im}\,e_1 +\bar k_0) e_2-4 \bar k_1 e_1-2 (n-4) (\bar k_0 k_1- k_0\bar k_1)-2 (\bar k_0 k_2- k_0\bar k_2)~.
\nonumber \\
\label{02cc2y}
\end{eqnarray}
\end{prop}

\begin{theorem}
\label{theorem02}
Let $M^n$ be a closed spin manifold.
Suppose that $F$ is closed, $ k_2= 4{{k}}_1+2i {\rm Im}\, e_2$ and ${\rm Im}\,c_3=0$. Then the following hold.
\begin{enumerate}

\item If ${1\over2} R+ c^{[0]}_1 f^2 +c^{[2]}_1 F^2\gneq 0$ and $c^{[2]}_2 \geq 0$, then ${\rm Ker}\,{\cal{D}}=\{0\}$~.

\item If ${1\over2} R+ c^{[0]}_1 f^2 +c^{[2]}_1 F^2= 0$, $c^{[2]}_2 =0$,  $e_1=n k_0$,
${\rm Im} \, e_2 \neq 0$,  and ${\rm Ker}\,{\cal{D}}\not=\{0\}$, then $e_2=(n-4) k_1+k_2$ and ${\rm Ker}\,{\cal{D}}={\rm Ker}\,{{\hat{\nabla}}}$.

\end{enumerate}

\end{theorem}
\begin{proof}
The proof of both assertions are similar to that given in the previous case.  It only remains to verify that both statements hold for a non-empty
range of parameters. Indeed eliminating $k_2$ from $c_2^{[2]}$ using  $ {{k}}_2= 4{{k}}_1+2i {\rm Im}\, e_2$, one finds that $c_2^{[2]}$ can be expressed as in (\ref{2c2}), where $e$ is replaced by $e_2$.  Similarly $c_3$ can be expressed as
\begin{eqnarray}
c_3=4(i {\rm Im}\,e_1 +\bar k_0) e_2-4 \bar k_1 e_1-2 n (\bar k_0 k_1- k_0\bar k_1)-8i   {\rm Re}\, k_0 {\rm Im}\, e_2~.
\end{eqnarray}
Clearly the condition ${\rm Im}\,c_3=0$ required for the validity of the first statement has many solutions.

Furthermore, the additional conditions that are required for the validity of the second statement give that
\begin{eqnarray}
{\rm Im}\, c_3=4n^2 {\rm Im}\, k_0\,\, {\rm Re}\, k_1 -4n {\rm Im}( k_0 \bar k_1)~.
\end{eqnarray}
  On the other hand
\begin{eqnarray}
c^{[2]}_2=2{n+1\over n} |e_2|^2-4 |e_2|^2\cos^2\psi=2n(n+1) |k_1|^2-4 n^2|k_1|^2\cos^2\psi ~,
\end{eqnarray}
as $e_2=n \bar k_1$ and $e_2=|e_2| \exp(i\psi)$.  The condition $c^{[2]}_2=0$ specifies the phase $\psi$. Then it is straightforward to observe that ${\rm Im}\, c_3=0$ has non-trivial solutions.
\end{proof}

\begin{prop}
\label{prop02}
Let $M^n$ be a closed spin manifold.
Suppose that $F$ is closed 2-form $dF=0$, $ k_2= 4{{k}}_1+2i {\rm Im}\, e_2$,  $k_0=e_1$, $k_1={\rm Re}\, e_2$,    $c_2^{[2]}\geq 0$ and ${\rm Im}\,c_3=0$. Then  the eigenvalues $\lambda$ of ${\cal{D}}$ are bounded as
\begin{eqnarray}
|\lambda|^2\geq \inf_{M^n} \left( {1\over4} R+ {1\over2} c^{[0]}_1 f^2 + {1\over2} c^{[2]}_1 F^2\right) \ .
\end{eqnarray}
\end{prop}

\begin{proof}
Again the proof is similar to those in previous cases. It remains that to prove that the statement is valid for a non-empty range of parameters. Indeed
\begin{eqnarray}
c_2^{[2]} &=& 2 |e_2|^2 (1-n \cos^2 \psi),
\nonumber \\
{\rm Im}\, c_3&=&4 |e_2| \big( (n-1) \cos \psi {\rm Im} \, e_1 - \sin \psi {\rm Re} \, e_1 \big) \ ,
\end{eqnarray}
where $e_2=|e_2| \exp(i\psi)$. So there is a range of parameters such that $c_2^{[2]} \geq 0$ and ${\rm Im}\,c_3=0$.

\end{proof}

\section{(0,3)-multi-form modified Dirac operators}

A (0,3)-multi-form modified Dirac operator  can be written as
\begin{eqnarray}
{\cal{D}}=\slashed{\nabla}+ e_1 f + e_2 \slashed H~,
\end{eqnarray}
where $e_1, e_2\in \mathbb{C}$,  and $f$ is a function and $H$ is a 3-form on $M^n$.  Furthermore consider
the connection
\begin{eqnarray}
{{\hat{\nabla}}}_X=\nabla_X+k_0 f \slashed X +k_1 \slashed H \cdot \slashed X+k_2 {i_X\slashed H}~,
\end{eqnarray}
on the spin bundle, where $k_0,  k_1, k_2\in \mathbb{C}$.

\begin{prop}
The fundamental identity is

\begin{eqnarray}
\nabla^2\parallel\epsilon\parallel^2&=&\left({1\over2} R+ c^{[0]}_1 f^2 +c^{[3]}_1 H^2\right)\parallel\epsilon\parallel^2+2 \parallel{{\hat{\nabla}}}\epsilon\parallel^2+c^{[3]}_2 \parallel\slashed H\epsilon\parallel^2
\cr
&+&2 {\rm Re} \langle\epsilon, (\slashed {\nabla}-2\bar k_0 f+\bar e_1 f+(-2\bar k_1+e_2)\slashed H){\cal{D}}\epsilon\rangle
+{\rm Re}\, \langle\epsilon, c_3 f \slashed H\epsilon\rangle
\cr
&+& {\rm Re} \langle \epsilon, \bigg((12(2 {\bar{k}}_1-{\rm Re}\,e_2)+4{\bar{k}}_2)H^i{}_{pq}\Gamma^{pq} \bigg) \nabla_i \epsilon \rangle
\cr
&-& 2 \nabla_i{\rm Re}\langle \epsilon, e_1 f \Gamma^i\epsilon\rangle
-{1 \over 2} {\rm Re} \langle \epsilon, e_2\,
{\slashed{dH}} \epsilon \rangle
-6\nabla^i {\rm Re} \langle \epsilon, e_2\,
H_{ijk} \Gamma^{jk} \epsilon \rangle
\nonumber \\
\label{03max}
\end{eqnarray}
where
\begin{eqnarray}
c^{[0]}_1&=&-2 |k_0|^2 n+2{\rm Re} ((2\bar k_0+e_1) e_1)-4({\rm Re}\,e_1)^2~,
\cr
c^{[3]}_1&=&-96 |k_1|^2-{8\over3}|k_2|^2-32 {\rm Re} (\bar k_2 k_1)~,
\cr
c^{[3]}_2&=&-{1\over9}\big(18 (n-8) |k_1|^2+2 |k_2|^2 -12 {\rm Re} (\bar k_2 k_1)\big)   +{\rm Re}((-4\bar k_1+2e_2)e_2)~,
\cr
c_3&=&-4 e_2 {\rm Re} \, e_1+4(\bar k_0 e_2+\bar k_1 e_1)-2 (6-n) (k_1\bar k_0-\bar k_1 k_0)
\cr
&&-2 (\bar k_0 k_2-k_0\bar k_2)~.
\nonumber \\
\end{eqnarray}

\end{prop}

\begin{theorem} \label{03th}
Let $M^n$ be a closed spin manifold.
Suppose that $H$ is closed 3-form $dH=0$, $ {{k}}_2= -6{{k}}_1+3{\rm Re}\, e_2$ and ${\rm Im}\,c_3=0$. Then the following hold.
\begin{enumerate}

\item If ${1\over2} R+ c^{[0]}_1 f^2 +c^{[3]}_1 H^2\gneq 0$ and $c^{[3]}_2 \geq 0$, then ${\rm Ker}\,{\cal{D}}=\{0\}$.

\item If ${1\over2} R+ c^{[0]}_1 f^2 +c^{[3]}_1 H^2= 0$, $c^{[3]}_2 =0$,  $e_1=n k_0$, and
${\rm Ker}\,{\cal{D}}\not=\{0\}$ then $e_2=(6-n)k_1+k_2$ and ${\rm Ker}\,{\cal{D}}={\rm Ker}\,{{\hat{\nabla}}}$.

\end{enumerate}

\end{theorem}
\begin{proof}
The proof of both statements is similar to those presented in previous case and so the steps will not be repeated here.  It remains
to demonstrate that there is a range of parameters such that the theorem holds. Indeed substituting $ {{k}}_2= -6{{k}}_1+3{\rm Re}\, e_2$ into
$c_3$, one finds that
\begin{eqnarray}
c_3=-4 e_2 {\rm Re}\, e_1+ 4(\bar k_0 e_2+\bar k_1 e_1)+2n (k_1 \bar k_0-k_0 \bar k_1)-6 (\bar k_0-k_0) {\rm Re}\, e_2 \ .
\end{eqnarray}
Clearly the conditions $c^{[3]}_2 \geq 0$ and ${\rm Im}\,c_3=0$ admit many solutions.

If in addition one imposes the assumption $e_1=n k_0$ and $e_2=(6-n)k_1+k_2$ on the parameters appearing in the second part of the theorem, then
\begin{eqnarray}
{\rm Im}\, c_3=4 n^2 {\rm Re}\, k_0\, {\rm Im}\, k_1-4n {\rm Im} (\bar k_0 k_1)
\end{eqnarray}
and
\begin{eqnarray}
c^{[3]}_2=-2n |e_2|^2\left( -{4\over n^2}+{3+n\over n^2}\sin^2\psi\right)~,
\end{eqnarray}
where $e_2=|e_2|\exp(i\psi)$.
Therefore the conditions ${\rm Im}\, c_3=c^{[3]}_2=0$ have solutions.
\end{proof}

\begin{corollary}\label{c:03f}
Let $M^n$ be a closed spin manifold.
Assuming the same conditions as those stated  in the first part of the theorem above, the eigenvalues $\tilde\lambda$ of the operator  $\tilde {\cal{D}}=\slashed \nabla+ e_2 \slashed H$, $e_2\in \mathbb{R}-\{0\}$,  are bounded as
\begin{eqnarray}
|\tilde\lambda|^2\geq  {1\over  s^2 n-2s+1} \,\inf_{M^n}\left( {1\over 4}R+{1\over2} c^{[3]}_1 H^2\right)~,
\label{s03bound}
\end{eqnarray}
where $s\in \mathbb{R}$ and $s>2n^{-1}$.
\end{corollary}

\begin{proof}   To derive this from the theorem \ref{03th} take $f=1$ and choose $e_1=-\tilde\lambda$ and $k_0=-s \tilde\lambda$, where $s\in \mathbb{R}$.  In which case, one has   ${\cal{D}}=\tilde {\cal{D}} -\tilde\lambda$. As $\tilde {\cal{D}}$ is formally anti-self-adjoint,  $\tilde\lambda$ is imaginary.  If $\epsilon$ is an eigen-spinor of $\tilde {\cal{D}}$ with eigen-value $\tilde\lambda$, then
${\cal{D}}\epsilon=0$.  As ${\cal{D}}$ has zero modes, it is required from the first part of the theorem above that
\begin{eqnarray}
{1\over2} R +c^{[0]}_1f^2+c^{[3]}_1 H^2
={1\over2} R +2|\tilde\lambda|^2 (-s^2 n+2s-1)+c^{[3]}_1 H^2\leq 0~.
\end{eqnarray}
As for $n>1$, $s^2 n-2s+1>0$,
this inequality can be re-arranged  to yield a bound  for the eigenvalues of $\tilde {\cal{D}}$ in (\ref{s03bound}).  An  upper lower bound for  the eigenvalues of $\tilde {\cal{D}}$ is
given for $s=n^{-1}$ which is the minimum of the polynomial.

It remains to investigate the range of parameters for which the corollary holds.  Indeed ${\rm Im}\,c_3=0$ for $\lambda\not=0$ gives
\begin{eqnarray}
{\rm Re}\, e_2={ns-1\over 2s} {\rm Re}\, k_1~.
\end{eqnarray}
Furthermore,
\begin{eqnarray}
c^{[3]}_2 = -2n ({\rm Im} \, k_1)^2 +2  ({\rm Re} \, k_1)^2 (n-2s^{-1})
\end{eqnarray}
and so the condition $c_2^{[3]}\geq 0$ holds for non-trivial values of $e_2$ provided that $s>2n^{-1}$.
\end{proof}

\begin{prop}
\label{prop03}
Let $M^n$ be a closed spin manifold.
Suppose that $dH=0$, ${{k}}_2= -6{{k}}_1+3{\rm Re}\, e_2$ and ${\rm Im}\,c_3=0$.  If $c^{[3]}_2\geq 0$,  $k_0= e_1$ and $k_1=-i{\rm Im}\, e_2$, then the eigenvalues
$\lambda$ of ${\cal{D}}$ are bounded as
\begin{eqnarray}
|\lambda|^2\geq \inf_{M^n} \left( {1\over4} R+ {1\over2} c^{[0]}_1 f^2 + {1\over2} c^{[3]}_1 H^2\right) \ .
\end{eqnarray}
\end{prop}
\begin{proof}
The conditions imposed on the parameters lead to
\begin{eqnarray}
c^{[3]}_2 = 2 (1-n) ({\rm Im} \, e_2)^2
\end{eqnarray}
and hence $c^{[3]}_2\geq 0$ implies that ${\rm Im} \, e_2=0$. With this choice,
\begin{eqnarray}
{\rm Im}\, c_3= 8 {\rm Re}\, e_2 {\rm Im}\, e_1~.
\end{eqnarray}
One chooses as a solution to ${\rm Im}\, c_3=0$ the condition  $ {\rm Im}\, e_1=0$.
\end{proof}

\section{Horizon Dirac operators}

Other examples of multi-form modified Dirac operators are the horizon Dirac operators. These are numerous and associated to supergravity theories
in all dimensions. We shall not provide details of how these appear and the applications they have in the context of black holes, see   \cite{mhor, revhor} for this. Here the focus is on the horizon Dirac operators of 11-dimensional supergravity. These are modified Dirac operators  with a (1,2,4)-multi-form.
There are two horizon Dirac operators defined on a 9-dimensional manifold $M^9$ given by
\begin{eqnarray}
\label{dirac1}
{\cal{D}}^{(\pm)}
= \slashed {\nabla}   \mp{1 \over 4} {\slashed {h}} +{1 \over 96} {\slashed {F}} \pm{1 \over 8} {\slashed {G}}
\end{eqnarray}
where $h$, $G$ and $F$ are 1-, 2- and 4-forms on $M^9$.  Furthermore define the covariant derivatives
\begin{eqnarray}
\label{identaux1}
{{\hat{\nabla}}}_X^{(\pm)}&=&\nabla_X\mp{1 \over 4} h(X) -{1 \over 288} {\slashed {F}} \cdot \slashed X +{1 \over 72}{ i_X \slashed F}
\cr
&\pm&{1 \over 24} {\slashed {G}} \cdot\slashed X\mp{1 \over 12} {i_X \slashed G}
\end{eqnarray}
on the real spinor bundle of ${\rm Spin}(9)$ on $M^9$.

\begin{prop}

The fundamental identity is
\begin{eqnarray}
\nabla^2 \parallel \epsilon_\pm \parallel^2
+{1 \over 2}(1 \mp 1) \nabla^i h_i  \parallel \epsilon_\pm \parallel^2
\mp h^i \nabla_i  \parallel \epsilon_\pm \parallel^2
\nonumber \\
= 2 \parallel {{\hat{\nabla}}}^{(\pm)} \epsilon_\pm,
\parallel^2
+{1 \over 2} {\cal R}  \parallel \epsilon_\pm \parallel^2
\mp {1 \over 2} \langle \epsilon_\pm ,{\slashed{\cal{G}}} \epsilon_\pm \rangle
-{1 \over 120} \langle \epsilon_\pm, {\slashed{dF}} \epsilon_\pm \rangle
\nonumber \\
+2 \langle \epsilon_\pm, \bigg({\slashed{\nabla}}
\mp {1 \over 4}{\slashed {h}} -{1 \over 288}{\slashed {F}} \mp {1 \over 24} {\slashed {G}}  \bigg){\cal{D}}^{(\pm)}\epsilon_\pm \rangle~,
\label{max11}
\end{eqnarray}
where
\begin{eqnarray}
{\cal{G}}_i=-(\delta G)_i-{}^*(F\wedge F)_i~,~~~
 {\cal R} = {{R}}-\delta h
-{1 \over 2}h^2-{1 \over 4} G^2-{1 \over 48}F^2~,
\end{eqnarray}
and $\epsilon_\pm$ are sections of the real spinor bundle of ${\rm Spin}(9)$ on $M^9$ which has rank 16.
\end{prop}

\begin{proof}
The fundamental identity can be established after some computation which is similar to those explained in the other cases. The only difference here
is that the spin bundle is real and so all the inner products are real.  This is why the reality restriction has been removed from the expressions.  This fundamental identity was originally derived in \cite{mhor} under the additional assumption that ${\cal R}={\cal G}=0$-these vanishing conditions are related to the field
equations of 11-dimensional supergravity.
\end{proof}

\begin{theorem}
Let $M^9$ be a closed 9-dimensional spin manifold.
Suppose that ${\cal G}=0$ and $dF=0$.
\begin{enumerate}

\item In addition if ${\cal R}+{1\pm1\over2} \delta h\gneq0$, then ${\rm Ker}\,{\cal{D}}^{(\pm)}=\{0\}$.

\item Furthermore if ${\cal R}+{1\pm1\over2} \delta h=0$, then ${\rm Ker}\,{\cal{D}}^{(\pm)}={\rm Ker}\,{{\hat{\nabla}}}^{(\pm)}$.

\end{enumerate}
\end{theorem}

\begin{proof}
To prove the first part of the statement assume that ${\cal{D}}^{(\pm)}$ has a non-trivial kernel and so there are non-trivial sections such that ${\cal{D}}^{(\pm)}\epsilon_\pm=0$. Using this and after imposing the assumptions of the theorem  on the fundamental identity, one is led to a contradiction after  integrating the resulting fundamental formula on $M^9$. Therefore ${\rm Ker}\,{\cal{D}}^{(\pm)}=\{0\}$.

To prove the second part assume  that $\epsilon_\pm\in {\rm Ker}\,{\cal{D}}^{(\pm)}$.  After imposing the assumptions of the theorem and integrating the fundamental formula over $M^9$, one finds that
$\epsilon_\pm\in {\rm Ker}\,{{\hat{\nabla}}}^{(\pm)}$.  Thus ${\rm Ker}\,{\cal{D}}^{(\pm)}\subseteq {\rm Ker}\,{{\hat{\nabla}}}^{(\pm)}$.  On the other hand a direct computation
reveals that ${\cal{D}}^{(\pm)}=\slashed\hn^{(\pm)}$. This establishes that ${\rm Ker}\,{\cal{D}}^{(\pm)}={\rm Ker}\,{{\hat{\nabla}}}^{(\pm)}$.
\end{proof}

{\bf Remark:} This illustrates that fundamental identities that lead to estimates for the eigenvalues of multi-form Dirac operators that contain degree $k\geq 4$ degree forms can be constructed.  The restriction to the $k\leq 3$-degree forms that we have found for $k$-form and (0,$k$)-form Dirac operators in not essential
and it can be removed. However the systematics of how such constructions can be done in general remain to be established.

\section{Concluding Remarks}

Multi-form modified Dirac operators have arisen  in the context of investigating black holes in supergravity and string theory, where some of their properties  have   been explored.  Here they  have been considered in the context of spin Riemannian manifolds of any dimension.  The focus has been on $k$-form and (0,$k$)-form modified Dirac operators. Some  results concerning  the horizon Dirac operators of \cite{mhor}, which are (1,2,4)-form modified Dirac operators, have also been included.
It has been found that many of the results that have been discovered for the standard Dirac operator extend to these multi-form modified Dirac operators for some range of parameters.  In particular, the Lichnerowicz formula and theorem  as well as the estimates for the eigenvalues of the Dirac
operator in \cite{fried, hijazi} suitably generalize.  In fact the estimates found in \cite{fried, hijazi} for the eigenvalues of the Dirac operator can be re-derived using the formalism developed for the multi-form Dirac operators.

Despite the progress that has been made, there is not a priori a procedure to find the multi-form modified Dirac operators that satisfy fundamental
identities which in turn can be used to relate their zero modes to parallel spinors and provide estimates for their eigenvalues. The construction we
have made applies only to certain cases and although extensions are possible, they depend more on a trial and error approach than on a systematic construction.
Nevertheless it is apparent that multi-form modified Dirac operators with fundamental identities that give rise to estimates for their eigenvalues
can be constructed on spin manifolds of any dimension.

Apart from providing  estimates for the eigenvalues multi-form modified Dirac operators, we have also investigated the conditions for the existence of parallel spinors with respect
to connections ${{\hat{\nabla}}}$ that are constructed from the frame connection and the multi-form that arises in the problem. We have demonstrated that the geometry of the
underlying manifold is associated with a twisted covariant form hierarchy which generalizes the Killing-Yano forms. In all cases, there is a range
of parameters for which the twisted covariant form hierarchy reduces to the conformal Killing-Yano forms even though ${{\hat{\nabla}}}$ and the associated multi-form Dirac
operator $\slashed \hn$ retain dependence on the multi-form. This is a significant simplification in the geometry of the manifold which does not have its origins
in the estimates for the eigenvalues of  the modified Dirac operator.  Although it is not apparent why such a simplification of the
twisted covariant form hierarchy is allowed, its presence signals the presence of a special structure which may require further investigation.
It is likely that the geometry of solutions of all  supergravity theories that  admit Killing spinors are associated with twisted covariant form hierarchies.
 This remark is expected to apply to all supergravity theories including
those defined  on Euclidean or non-Lorentzian spacetime signature spacetimes.

As an application to geometry, the question arises whether the Gromov-Lawson theorem can be adapted to the context of spin Riemannian manifolds equipped
with a closed 3-form.  We have seen that the condition $ R-48({\rm Re}\,e)^2 H^2\gneq  0$ that arises in the theorem \ref{th:3f} requires that the $\alpha$-invariant of ${\cal{D}}$, which is the same as that of the Dirac operator, vanishes. So the vanishing of the $\alpha$-invariant is a necessary condition for the existence of a metric and a closed 3-form
such that $ R-48({\rm Re}\,e)^2 H^2\gneq  0$. However, it is not apparent that only the topological condition suffices to establish the converse, even in the case of simply connected manifolds. It is likely  that additional conditions are  required to be explored elsewhere in the future.

\section*{Acknowledgments}

  JG is supported by the STFC Consolidated Grant ST/L000490/1. JG would like to thank the Department of Mathematics, University of Liverpool, for hospitality during which part of this work was completed.

\setcounter{section}{0}\setcounter{equation}{0}

\appendix{Notation}

Let $(M^n,g)$ be a  closed spin Riemannian manifold with metric $g$. Forms are normalized in the standard way. In particular in a co-frame basis $\{e^i;i=1,\dots,n\}$ a $p$-form $\omega$ is expressed as
\begin{eqnarray}
\omega={1\over p!} \omega_{i_1\dots i_p} e^{i_1}\wedge\dots \wedge e^{i_p}~.
\end{eqnarray}
The point-wise inner product of two $p$-forms is
\begin{eqnarray}
\langle \chi, \omega\rangle\mathrel{\vcenter{\baselineskip0.5ex \lineskiplimit0pt
                     \hbox{\scriptsize.}\hbox{\scriptsize.}}}
                    ={1\over p!} \chi_{i_1\dots i_p} \omega^{i_1\dots i_p}~,
\end{eqnarray}
and  the point-wise norm is
\begin{eqnarray}
|\omega|^2\mathrel{\vcenter{\baselineskip0.5ex \lineskiplimit0pt
                     \hbox{\scriptsize.}\hbox{\scriptsize.}}}
                    =\langle \omega, \omega\rangle\mathrel{\vcenter{\baselineskip0.5ex \lineskiplimit0pt
                     \hbox{\scriptsize.}\hbox{\scriptsize.}}}
                    ={1\over p!} \omega^2~,
\end{eqnarray}
where the indices are raised with respect to a metric on $M^n$.  Notice that $\omega^2\mathrel{\vcenter{\baselineskip0.5ex \lineskiplimit0pt
                     \hbox{\scriptsize.}\hbox{\scriptsize.}}}
                    =\omega_{i_1\dots i_p}\omega^{i_1\dots i_p}$, i.e. the expression does not have a  normalization
factor.

The inner derivation of a $p$-form $\omega$ with a vector $q$-form $L$, $i_L\omega$, is given by
\begin{eqnarray}
i_L\omega\mathrel{\vcenter{\baselineskip0.5ex \lineskiplimit0pt
                     \hbox{\scriptsize.}\hbox{\scriptsize.}}}
                    ={p\over p! q!}\, L^k{}_{i_1\dots i_q} \omega_{k i_{q+1}\dots i_{p+q-1}} \,e^{i_1}\wedge\dots\wedge e^{i_{p+q-1}}~.
\end{eqnarray}
We also define the inner derivation of a $p$-form $\omega$ with a $\ell$-form $\chi$ as
\begin{eqnarray}
i_\chi \omega\mathrel{\vcenter{\baselineskip0.5ex \lineskiplimit0pt
                     \hbox{\scriptsize.}\hbox{\scriptsize.}}}
                    ={p\over p! (\ell-1)!} \chi^k{}_{i_1\dots i_{\ell-1}} \omega_{k i_{\ell}\dots i_{p+\ell-2}} \,e^{i_1}\wedge\dots\wedge e^{i_{p+\ell-2}}~,
\end{eqnarray}
where the index is raised with respect to the metric.
In general we use the same symbol to denote tensors which are related by raising and lowering indices with respect to the metric to simplify the formulae.
It is apparent from the context which kind of tensor is used each time.

The adjoint of the inner derivation of a $p$-form $\omega$ with respect to the vector $q$-form $L$, $i^\dagger_L\omega$, is defined as
\begin{eqnarray}
\langle\phi, i^\dagger_L\omega\rangle\mathrel{\vcenter{\baselineskip0.5ex \lineskiplimit0pt
                     \hbox{\scriptsize.}\hbox{\scriptsize.}}}
                    =\langle i_L\phi, \omega\rangle
\end{eqnarray}
for any $(p-q+1)$-form $\phi$. $i^\dagger_\chi \omega$ can be defined in a similar way.

Similarly the adjoint operation, $\vee$, of the wedge product, $\wedge$, of a $p$-form $\omega$ with a $q$-form $\chi$, $q>p$, is
\begin{eqnarray}
\langle\phi, \omega\vee \chi\rangle\mathrel{\vcenter{\baselineskip0.5ex \lineskiplimit0pt
                     \hbox{\scriptsize.}\hbox{\scriptsize.}}}
                    =\langle\omega\wedge\phi, \chi\rangle~
\end{eqnarray}
for an every $q-p$ form $\phi$.
The exterior derivative $d\omega$ of a $p$-form $\omega$ is defined in the standard way.  The adjoint $\delta$ of $d$ is defined as
$ \langle \delta\omega, \chi\rangle_{M^n}\mathrel{\vcenter{\baselineskip0.5ex \lineskiplimit0pt
                     \hbox{\scriptsize.}\hbox{\scriptsize.}}}
                    = \langle \omega, d\chi\rangle_{M^n}$, where $\chi$ is any ($p$-1)-form and $\langle \cdot, \cdot\rangle_{M^n}\mathrel{\vcenter{\baselineskip0.5ex \lineskiplimit0pt
                     \hbox{\scriptsize.}\hbox{\scriptsize.}}}
                    =
\int_{M^n} \langle \cdot, \cdot\rangle$
is the integrated inner product of forms over the manifold $M^n$.

The bundle isomorphism $/~:~ \Lambda^*(M^n)\rightarrow {\rm Cl}(M^n)$  between the bundle of forms $\Lambda^*(M^n)$  and
the Clifford algebra bundle ${\rm Cl}(M^n)$ is defined as
\begin{eqnarray}
/(\omega)\mathrel{\vcenter{\baselineskip0.5ex \lineskiplimit0pt
                     \hbox{\scriptsize.}\hbox{\scriptsize.}}}
                    = \slashed\omega\mathrel{\vcenter{\baselineskip0.5ex \lineskiplimit0pt
                     \hbox{\scriptsize.}\hbox{\scriptsize.}}}
                    =\omega_{i_1\dots i_p} \Gamma^{i_1}\cdots\Gamma^{i_p}~,
\end{eqnarray}
on a $p$-form $\omega$, where $\cdot$ denotes Clifford algebra multiplication and $\{\Gamma^{i};i=1,\dots, n\}$ is a basis in the Clifford algebra associated with the co-frame $\{e^i;i=1,\dots,n\}$ on $M^n$.  Note that the normalization factor in front of the form has been dropped.  Similarly one can define $\slashed X$ for $X$ a vector field and the Clifford algebra relation reads $\slashed X\cdot \slashed Y+\slashed Y\cdot \slashed X=2 g(X,Y) 1$.

The point-wise spin invariant hermitian (Dirac) inner product of spinors $\eta$ and $\epsilon$ is denoted as $\langle \eta, \epsilon\rangle$ and the associated norm as $ \parallel\epsilon\parallel^2 =\langle \epsilon, \epsilon\rangle$. In these conventions $\langle \slashed X\eta, \epsilon\rangle=\langle \eta, \slashed X\epsilon\rangle$ for every vector field $X$. (In many texts $\slashed X\cdot \slashed Y+\slashed Y\cdot \slashed X=-2 g(X,Y) 1$ and $\langle \slashed X\eta, \epsilon\rangle=-\langle \eta, \slashed X\epsilon\rangle$ but we do not follow these conventions.)

\end{document}